\documentclass[12pt,a4paper, reqno]{amsart}
\usepackage{amsfonts}
\usepackage{amsmath}
\usepackage{amssymb}
\usepackage{latexsym}
\usepackage{exscale}

\usepackage[colorlinks=true, pdfstartview=FitH, linkcolor=blue, citecolor=red, urlcolor=blue]{hyperref}

\allowdisplaybreaks

\newtheorem{theorem}{Theorem}[section]
\newtheorem{corollary}[theorem]{Corollary}
\newtheorem{proposition}[theorem]{Proposition}
\newtheorem{lemma}[theorem]{Lemma}

\theoremstyle{definition}\newtheorem{definition}[theorem]{Definition}
\theoremstyle{remark}\newtheorem{remark}[theorem]{Remark}
\theoremstyle{remark}\newtheorem{example}{Example}

\numberwithin{equation}{section}

\newcommand{\R}{\mathbb{R}}
\newcommand{\re}{\mathbb{R}}
\newcommand{\co}{\mathbb{C}}

\def\div{\mathop{\operatorname{div}}}

\newcommand{\Qcal}{\mathcal{Q}}
\newcommand{\F}{\mathcal{F}}

\newcommand{\Id}{\mathcal{I}}

\renewcommand{\emptyset}{\textup{\mbox{\O}}}

\DeclareMathOperator*{\essinf}{ess\,inf}

\def\Xint#1{\mathchoice
   {\XXint\displaystyle\textstyle{#1}}%
   {\XXint\textstyle\scriptstyle{#1}}%
   {\XXint\scriptstyle\scriptscriptstyle{#1}}%
   {\XXint\scriptscriptstyle\scriptscriptstyle{#1}}%
   \!\int}
\def\XXint#1#2#3{{\setbox0=\hbox{$#1{#2#3}{\int}$}
     \vcenter{\hbox{$#2#3$}}\kern-.5\wd0}}

\def\aver#1{\Xint-_{#1}}

\author{Fr\'ed\'eric Bernicot}
\address{Fr\'ed\'eric Bernicot \\ Laboratoire de Math\'ematiques Jean Leray \\ 2, Rue de la Houssini\`ere F-44322 Nantes Cedex 03, France.}
\email{frederic.bernicot@univ-nantes.fr}

\author{Jos\'e Mar{\'\i}a Martell}
\address{Jos\'e Mar{\'\i}a Martell
\\
Instituto de Ciencias Matem\'aticas CSIC-UAM-UC3M-UCM
\\
Consejo Superior de Investigaciones Cient{\'\i}ficas
\\
C/ Nicol\'as Cabrera, 13-15
\\
E-28049 Madrid, Spain} \email{chema.martell@icmat.es}

\thanks{The second author was supported by MINECO Grant MTM2010-16518 and ICMAT Severo Ochoa project SEV-2011-0087. Both authors wish to thank
Pascal Auscher, Steve Hofmann and Svitlana Mayboroda for  helpful comments concerning some of the applications.}

\date{\today}


\begin{document}
\title{Self-improving properties for abstract Poincar\'e type inequalities}


\begin{abstract}
We study self-improving properties in the scale of Lebesgue spaces of generalized Poincar\'e inequalities in the Euclidean space. We present an abstract setting where oscillations are given by certain operators (e.g., approximations of the identity, semigroups or mean value operators) that have off-diagonal decay in some range. Our results provide a unified theory that is applicable to the classical Poincar\'e inequalities and furthermore it includes oscillations defined in terms of semigroups associated with second order elliptic operators as those in the Kato conjecture. In this latter situation we obtain a direct proof of the John-Nirenberg inequality for the associated $BMO$ and Lipschitz spaces of \cite{Hof-May, Hof-May-Mc}.
\end{abstract}

\newcommand{\status}[1]{\vskip-1.2cm\noindent\parbox{\textwidth}{%
\hfill \texttt{#1}}\vskip.2cm\null}

\status{Trans.~Amer.~Math.~Soc.~367, no.~7, (2015), 4793-4835. }

\maketitle

\begin{quote}
\footnotesize\tableofcontents
\end{quote}

\section{Introduction}

There are various inequalities in analysis that encode intrinsic self-improving properties of the oscillations of the functions involved. For instance, the classical John-Nirenberg inequality (see \cite{JN}) states that for every $f\in L_{\rm loc}^1(\re^n)$ such that
$$
\aver{Q} |f-f_Q|\,dx \leq C
$$
for every cube $Q$, (i.e., $f\in BMO$),
the oscillations $|f-f_Q|$ are exponentially integrable. In particular, for every $1<p<\infty$,
$$
\aver{Q} |f-f_Q|^p \,dx \leq C.
$$
Therefore, the oscillations $|f-f_Q|$, which are a priori in $L^1(Q)$, turn out to be in the ``better'' space $L^p(Q)$. The same occurs with the well-known fact that $(1,1)$-Poincar\'e inequality in $\R^n$, $n\ge 2$,
$$
\aver{Q} |f-f_Q|\,dx \leq C\,r_Q\,
\aver{Q}|\nabla f|\,dx,
$$
implies that, for all $1\le p<n$,
$$
\left(\aver{Q}\left|f-f_Q \right|^{p^*}\,dx\right)^{1/p^*}
\leq
C\,r_Q \left(\aver{Q} |\nabla f|^p\, dx\right)^{1/p},
$$
where $p^*=\frac{p\,n}{n-p}$. Again, if $f$ is such that $\nabla f\in L^p_{\rm loc}(\re^n)$, the fact that the oscillation is in $L^p(Q)$ yields that the oscillation is indeed in the smaller space $L^{p*}(Q)$. B. Franchi, C. P\'erez and R.L. Wheeden in \cite{Franchi-Perez-Wheeden} gave a unified approach to these kinds of estimates. Namely,  they start with inequalities of the form
\begin{equation}\label{des:Poincare generalizada}
\int \left|f- f_Q \right|\, dx \leq a(Q,f),
\end{equation}
where $a$ is a functional depending on the cube $Q$, and sometimes on the function $f$. Using the Calder\'on-Zygmund theory
and the good-$\lambda$ inequalities introduced by D.L. Burkholder and R.F. Gundy \cite{BG} these authors present a general method that gives $L^p(Q)$  integrability of the oscillation $|f-f_Q|$ under mild geometric conditions on the functional $a$. Thus, inequality \eqref{des:Poincare generalizada} encodes an
intrinsic self-improvement on $L^p$ for $p>1$.

Generalizations of the previous estimates have been already studied. One can define new oscillations by replacing the averaging operator $f_Q$ by some other operator $A_{t_Q} f$ (here $t_Q$ is a parameter defined in terms of the sidelength of $Q$ that ``scales'' $A_t$ to $Q$) where $\mathbb{A}:=(A_t)_{t>0}$ is a semigroup (or, more in general, some ``approximation of the identity''). In the particular case where the kernels of the operators  have enough decay (i.e., the family $\mathbb{A}$ satisfies $L^1-L^\infty$ off-diagonal estimates) this new way of measuring the oscillation allowed the second author \cite{martell-04} to introduce a new sharp maximal operators by simply replacing $|f-f_Q|$ by $|f-A_{t_Q}f|$. This gave raise to a new $BMO$ space introduced by X.T. Duong and L. Yan \cite{DY-CPAM}, \cite{DY-JAMS} which also enjoys the John-Nirenberg inequality. This ultimately says that an $L^1(Q)$ estimate for $|f-A_{t_Q}f|$ implies $L^p(Q)$ estimates for these new oscillations. When the family $\mathbb{A}$ is a semigroup generated by a second order divergence form elliptic operator with complex coefficients, S. Hofmann and S. Mayboroda in \cite{Hof-May} establish the John-Nirenberg inequality for the associated $BMO$ space using an indirect argument passing through the Carleson measure characterization of that space. We also refer the reader to \cite{BZ3} for a recent work where the first author and J. Zhao obtain such properties with more abstract operators and describe some applications to Hardy spaces.

The study of generalized Poincar\'e inequalities {\it \'a la} \cite{Franchi-Perez-Wheeden} for these new oscillations associated with a family $\mathbb{A}$ with $L^1-L^\infty$ decay has been recently studied by N. Badr, A. Jim\'enez-del-Toro and the second author in the papers \cite{Jim, JM, BJM}. The exponential self-improving (which corresponds to the case where the functionals $a$ above are quasi-increasing) is considered in \cite{Jim} and covers in particular the John-Nirenberg inequality for the $BMO$ space of \cite{DY-CPAM}. The $L^p$ self-improving is thoroughly studied in \cite{JM, BJM} where different applications to generalized (pseudo) Poincar\'e inequalities with oscillations $|f-A_{t_Q} f|$ are given.

In this paper, continuing the previous line of research we present a final result to the study of self-improving of generalized Poincar\'e inequalities. The main goal is to describe an abstract setting containing existing results for oscillations based both on the average operators and on an approximation of the identity satisfying off-diagonal estimates.
More precisely, our main result is written in an abstract way with oscillations depending on a family ${\mathbb A}:=(A_Q)_{Q}$ that is indexed by the cubes and that satisfy off-diagonal estimates in some range $[p_0,q_0]$ with $1\le p_0<q_0\le\infty$. These estimates contain the ``cancelation'' provided by the compositions $(1-A_{Q_1}) A_{Q_2}$ for cubes  $Q_1 \subset Q_2$ and this is one of the original ideas in the paper.

We can recover the results in \cite{Franchi-Perez-Wheeden} where $p_0=1$, $q_0=\infty$ and $A_Q$ is essentially the averaging operator on $Q$ ---here the fact that the operators are allowed to depend on $Q$ in place of its sidelength is crucial---  and also the results in \cite{Jim,JM} where again $p_0=1$, $q_0=\infty$ and $A_Q=A_{t_Q}$ as explained above. Besides, our results are applicable to semigroups $e^{-t\,L}$ with $L$ being a second order divergence form elliptic operator with complex coefficients. In such a case we obtain a direct proof of the John-Nirenberg inequality for the $BMO$ space defined in \cite{Hof-May}. We also obtain similar estimates for the corresponding Lipschitz spaces introduced in \cite{Hof-May-Mc} and this allows us to define these spaces using $L^p$ averages (in place of $L^2$ as done in \cite{Hof-May-Mc}) of the oscillations. Also our results can be applied to derive some generalized Poincar\'e inequalities with right hand sides that are dyadic expansions taking into account the lack of localization of the semigroups.

The plan of the paper is as follows. In Section \ref{section:preliminaries} we present some preliminaries. Section \ref{section:main-result} contains the main result Theorem \ref{theorem:Lp} and its generalizations. Among them we point out Theorems \ref{theorem:Lp-alternative} and \ref{theorem:Lp-alternative-bis} where we remove the commutative condition assumed on the oscillation operators and in the latter we further impose some localization property that allows us to recover the results from \cite{Franchi-Perez-Wheeden}. Additionally Theorem \ref{theorem:Lp-w} contains some extension of the main result with Muckenhoupt weights in the estimates. Applications are presented in Section \ref{section:appl}. We first consider John-Nirenberg inequalities for general $BMO$ spaces. We also propose some functionals and oscillation operators that fulfill the required hypotheses. All these are applied to the case of second order divergence form elliptic operators obtaining the new direct proof of the John-Nirenberg inequality for the associated $BMO$ space from \cite{Hof-May}, the new analogous result for the associated Lipschitz spaces from \cite{Hof-May-Mc} and some ``expanded'' Poincar\'e type inequalities. Finally Section \ref{sec:proof} is devoted to the proofs of the main and auxiliary results.

\section{Preliminaries}\label{section:preliminaries}

Let us consider the Euclidean space $\R^n$  with the Lebesgue measure $dx$ and the distance $|x-y|=|x-y|_\infty=\max_{1\le i\le n}|x_i-y_i|$. Without loss of generality, we assume that the cubes are of the form
$\prod_{i=1}^n \big[a_i, a_i + \ell(Q)\big)$, with $a_i\in \R$ and where $\ell(Q)$ denotes the sidelength of $Q$. Given a cube $Q\subset\R^n$ we denote its center by $x_Q$ and its sidelength by $\ell(Q)$. For any $\lambda>1$, we denote by $\lambda\,Q$ the cube concentric with $Q$ so that $\ell(\lambda\,Q)=\lambda\,\ell(Q)$.
We write $L^p$ for $L^p(\R^n,\R)$ or $L^p(\R^n,{\mathbb C})$.
The average of $f\in L^1(Q)$ in $Q$ is denoted by
$$
f_Q=\aver{Q} f(x)\,dx
=
\frac1{|Q|}\,\int_{Q} f(x)\,dx.
$$
The localized and normalized ``norm'' of a Banach or a quasi-Banach function space $X$ is written as
$$
\|f\|_{X,Q}=\|f\|_{X (Q, dx/|Q|)}
$$
Examples of spaces $X$ are $L^{p,\infty}$, $L^p$ or more general
Marcinkiewicz and Orlicz spaces.

Let us denote by $\Qcal$ the collection of all cubes in $\re^n$. We write $M$ for the maximal Hardy-Littlewood function:
$$
M f(x)= \sup_{{\genfrac{}{}{0pt}{}{Q\in\Qcal}{x\in Q}}}  \aver{Q}|f|dx.$$
For $p\in [1,\infty)$, we set $M_p f(x)=M(|f|^p)(x)^{1/p}$.

\subsection{Oscillation operators}
Let $\mathbb{B}:=(B_Q)_{Q\in \Qcal}$ be a family of linear operators indexed by the collection $\Qcal$.
The reader may find convenient to think of $B_Q$ as being some kind of ``oscillation operator'' on the cube $Q$. For each cube $Q$, we set $A_Q:=I-B_Q$ and $A_Q$ could be thought as an approximation of the identity.

\begin{example} The classical oscillation operator is defined as
$$
B_Q f:= f - \left(\aver{Q} f\,dx\right) \chi_Q.
$$
In this case, $A_Q f=\big(\aver{Q} f\,dx \big) \chi_Q$ is the averaging operator on $Q$. We consider a variant of this example in Section \ref{subsec:comm} below.
\end{example}

\begin{example}
Given  a differential operator $-L$ that generates a semigroup $e^{-t\,L}$, we can choose the oscillations
$$
B_Q f := f - e^{-\ell(Q)^mL}f
\qquad \mbox{or}\qquad
B_Q f := (I - e^{-\ell(Q)^mL})^N f
$$
depending on the order $m$ of $L$ and with $N\ge 1$.
\end{example}

We refer the reader to \cite[Section 3]{BZ} for some other oscillation operators and to \cite{BZ2} for a specific example applied to the problem of maximal regularity. Let us notice that a different notation is used in \cite{Jim, JM, BJM}, where there are some family of operators $S_t$ that play the role of generalized approximations of the identity under the assumption that the kernels decay fast enough and the operators commute. In such a case we can define $A_Q= S_{t_Q}$ and $B_Q=I-S_{t_Q}$ where $t_Q=\ell(Q)^m$  for some positive constant $m$.

\begin{definition} \label{def:Off}
Let $1\le p_0\le q_0\le \infty$ and $\mathbb{B}:=(B_Q)_{Q\in \Qcal}$ be as before. We say that $\mathbb{B}$ is $O(p_0,q_0)$ if the following conditions hold:
\begin{list}{$(\theenumi)$}{\usecounter{enumi}\leftmargin=.8cm
\labelwidth=.8cm\itemsep=0.2cm\topsep=.1cm
\renewcommand{\theenumi}{\alph{enumi}}}

\item The operators $B_Q$ commute:  $B_Q\,B_R=B_R\,B_Q$ for every $Q, R\in\Qcal$.

\item The operators $B_Q$ are uniformly bounded on $L^{p_0}(\re^n)$:
\begin{equation} \label{BQ:Lp-uniform}
\|B_Q f \|_{L^{p_0}(\re^n)} \leq C_{p_0}\,\|f\|_{L^{p_0}(\re^n)},
\qquad \mbox{for all } Q\in\Qcal.
\end{equation}

\item The operators $A_Q$ satisfy $L^{p_0}-L^{q_0}$ off-diagonal estimates at the scale $Q$: there exist fast decay coefficients $\alpha_k$, $k\ge 2$, such that for all cubes $Q$,
we have
\begin{equation} \label{AQ:on:p-q}
\left(\aver{2Q} \left|A_Q (f\,\chi_{4\,Q})\right|^{q_0}\,dx\right)^{1/q_0} \leq \alpha_2 \left(\aver{4Q} |f|^{p_0}\,dx \right)^{1/p_0}.
\end{equation}
and for all $j\geq 1$,
\begin{equation} \label{AQ:off:p-q}
\left(\aver{2^{j}Q} \left|A_Q (f\,\chi_{\re^n\setminus 2^{j+1}\,Q})\right|^{q_0}\,dx\right)^{1/q_0} \leq \sum_{k\geq 2} \alpha_{k+j} \left(\aver{2^{k+j} Q} |f|^{p_0}\,dx \right)^{1/p_0}.
\end{equation}

\item The operators $B_R\,A_Q$ satisfy $L^{p_0}-L^{q_0}$ off-diagonal estimates at the lower scale:  there exist fast decay coefficients $\beta_k$, $k\ge 2$, such that for all cubes $R\subset Q$,
\begin{equation}\label{BR-AQ:off:p-q}
\left(\aver{2R} \left|B_R A_Q f\right|^{q_0}\,dx \right)^{1/q_0} \leq \sum_{k\geq 1} \beta_{k+1} \left(\aver{2^{k+1} Q} |f|^{p_0}\,dx \right)^{1/p_0}.
\end{equation}

\end{list}
\end{definition}

Let us observe that when $q_0 = \infty$ one has to change the $L^{q_0}$-norms by the corresponding essential suprema. Notice also that by Jensen's inequality $O(p_0,q_0)$ implies $O(p_1,q_1)$ for all $p_0\le p_1\le q_1\le q_0$. Let us observe that \eqref{AQ:on:p-q} and \eqref{AQ:off:p-q} with $j=1$ yield
\begin{equation}\label{AQ:off:p-q:j=1}
\left(\aver{2Q} \left|A_Q f\right|^{q_0}\,dx\right)^{1/q_0} \leq \sum_{k\geq 2} \alpha_k \left(\aver{2^{k} Q} |f|^{p_0}\,dx \right)^{1/p_0}.
\end{equation}
Besides for all $j\ge 1$, Jensen's inequality, \eqref{BQ:Lp-uniform} and \eqref{AQ:off:p-q} imply the following $L^{p_0}-L^{p_0}$ off-diagonal estimates:
\begin{equation}\label{AQ:off:p-p}
\left(\aver{2^j Q} \left|A_Q f\right|^{p_0}\,dx\right)^{1/p_0} \leq
C\,\left(\aver{2^{j+1} Q} |f|^{p_0}\,dx\right)^{1/p_0} +
\sum_{k\geq 2} \alpha_{k+j} \left(\aver{2^{k+j} Q} |f|^{p_0}\,dx \right)^{1/p_0}.
\end{equation}

\begin{remark}\label{remark:q0-infty}
When $q_0=\infty$ one can show that $(d)$ follows from $(b)$ and $(c)$, details are left to the reader.
\end{remark}

\begin{remark}
In Definition \ref{def:Off}, it is implicitly assumed that the operators $B_Q$ and $A_Q$ are well defined for functions in $L^{p_0}_{\rm loc}(\re^n)$, we just need to write $f=\sum_{l=2}^\infty f_l$
with $f_1=f\,\chi_{4\,Q}$ and $f_l=2^{l+1}Q\setminus 2^l\,Q$ for $l\ge 2$. Note that $(b)$ implies that $B_Q f_l$, $A_Q f_l$ are in $L^{p_0}(\re^n)$. Furthermore \eqref{AQ:off:p-q:j=1} yields $A_Q f_l, B_Q f_l\in L^{q_0}(2Q)$.
\end{remark}

\subsection{Functionals}

We consider functionals
$$
a:\Qcal\times\mathcal{F}\longrightarrow[0,\infty),
$$
where $\mathcal{F}$ is a certain family of functions in $L^{p_0}_{\rm loc}(\re^n)$. When the dependence on the functions is not of our interest, we simply write $a(Q)$. Next, we define the geometric conditions $D_r$, first introduced in
\cite{Franchi-Perez-Wheeden}, \cite{MP} to study self-improving properties of generalized Poincar\'e inequalities associated to classical oscillations.

\begin{definition}\label{def:Dr}
Let $\mu$ be a Borel measure and let $a$ be  a functional as before.
\begin{list}{$(\theenumi)$}{\usecounter{enumi}\leftmargin=.9cm
\labelwidth=.9cm\itemsep=0.2cm\topsep=.1cm
\renewcommand{\theenumi}{\alph{enumi}}}

\item Given $1\le r<\infty$, we say that $a$ satisfies the $D_r(\mu)$ condition or $a\in D_r(\mu)$ (if $\mu$ is the Lebesgue measure we simply write $D_r$) if there exists a finite constant $C_{a}\ge 1$ such that for each cube $Q$ and any family $\{Q_i\}_i\subset Q$ of pairwise disjoint cubes,
$$
\sum_{i} a(Q_i)^r\, \mu(Q_i)\leq C_{a}^r\, a(Q)^r\, \mu(Q).
$$
The infimum of the constants $C_{a}$ is denoted by $\|a\|_{D_r(\mu)}$.

\item We say that $a$ satisfies the $D_\infty$ condition or $a\in D_\infty$ if $a$ is quasi-increasing, that is, there exists a constant $C_{a}\ge 1$ such that for all cubes $R\subset Q$,
$$
a(R)\le C_{a}\, a(Q).
$$
The infimum of the constants $C_{a}$ is denoted by $\|a\|_{D_\infty}$.

\item We say that $a$ is doubling if there exists a constant $C_a>0$ such that for every cube $Q$,
$$
 a(2\,Q) \leq C_a\, a(Q).
$$

\item We say that $a\in D_0$ if it is locally quasi-increasing, that is, there exists a constant $C_{a}\ge 1$ such that for all cubes $R\subset Q$ with $\ell(Q)\le 4\,\ell(R)$,
$$
a(R)\le C_{a}\, a(Q).
$$

\end{list}
\end{definition}

As an immediate consequence of H\"older's inequality,
one sees that the $D_r(\mu)$ conditions are decreasing. That is,
if $1\leq  r < s < \infty$, then
$D_s(\mu)\subset D_r(\mu)$ and $\|a\|_{D_r(\mu)}\leq \|a\|_{D_s(\mu)}$.
This property will be used several times during the proofs. Note also that if $a\in D_\infty$ then $a\in D_r(\mu)$ with $\|a\|_{D_r(\mu)}\leq \|a\|_{D_\infty}$ for any Borel measure $\mu$ and $1\leq r<\infty$. We also notice that $D_1(\mu)\subset D_0$ for any Borel doubling measure $\mu$: if $Q\subset R$ is such that $\ell(R)\le 4\,\ell(Q)$ we have $R\subset 8\,Q$ and therefore
\begin{multline*}
a(Q)
=
\mu(Q)^{-1}\,a(Q)\,\mu(Q)
\le
\|a\|_{D_1(\mu)}\,\mu(Q)^{-1}\,a(R)\,\mu(R)
\\
\le
\|a\|_{D_1(\mu)}\,\mu(Q)^{-1}\,a(R)\,\mu(8\,Q)
\le
\|a\|_{D_1(\mu)}\,C_\mu a(R).
\end{multline*}
All these together yield that if $\mu$ is a doubling Borel measure, then for every $1\le r\le  s<\infty$ we have
$$
D_\infty \subset D_s(\mu)\subset D_r(\mu)\subset D_1(\mu)\subset D_0.
$$

\section{Main results}\label{section:main-result}

In what follows $\mathcal{F}$ is a given family of functions in $L^{p_0}_{\rm loc}(\re^n)$ (for a more general family $\mathcal{F}$ see Remark \ref{remark:distributions} below). We now state our main result.

\begin{theorem}\label{theorem:Lp}
Fix $1\le p_0< q<q_0 \le \infty$ and $\mathbb{B}:=(B_Q)_{Q\in \Qcal}$ in $O(p_0,q_0)$.
Given $f\in \mathcal{F}$ and a functional $a$, let us assume that
for every $k\geq 0$ and every cube $Q$,
\begin{equation}\label{hyp:BQ:2k-Q}
\left(\aver{2^k Q} |B_Q f|^{p_0} \right)^{1/p_0} \leq a(2^k Q).
\end{equation}
There exist fast decay coefficients $\{\tilde{\gamma}_k\}_{k\ge 1}$  such that if we define
\begin{equation} \label{eq:tilde}
\tilde{a}(Q)
=
\sum_{k=1} ^\infty \tilde{\gamma}_k\,a(2^k\,Q)
\end{equation}
and $\tilde{a}\in D_q$, then for every cube $Q$,
\begin{equation}\label{conclusion:theor-Lp}
\|B_Q f\|_{L^{q,\infty}, Q}
\lesssim
\tilde{a}(2\,Q).
\end{equation}
\end{theorem}

\bigskip

\begin{remark}
When the operators $B_Q=B_{\ell(Q)}$ are given by semigroups or approximations of the identity admitting an integral representation with kernels that have enough pointwise decay we recover the results in \cite{JM}. This case corresponds to $p_0=1$ and $q_0=\infty$.
\end{remark}

\begin{remark} \label{rem:main-doubling}
Let us observe that if $a$ is doubling (and the sequences in $O(p_0,q_0)$ decay fast enough), it then suffices to assume that $a\in D_q$ (in place of $\tilde{a}\in D_q$) and we conclude that
$$
{\|B_Q f\|}_{L^{q,\infty}, Q} \lesssim a(Q).
$$
This follows from the general result since we easily obtain $a\approx \tilde{a}$.
\end{remark}

\begin{remark} \label{remark:side-Q}
When the operators $B_Q=B_{\ell(Q)}$ only depend on the sidelength of the cube and one further assumes $a\in D_{p_0}$, then \eqref{hyp:BQ:2k-Q} (with a constant in front of $a(2^k\,Q)$) is a direct consequence of the case $k=0$  and $a\in D_{p_0}$. To see this, we decompose $2^kQ$ as a union of disjoint cubes with sidelength $Q$, apply the case $k=0$ and use the condition $a\in D_{p_0}$ to sum over those cubes (see \cite[Lemma 5.1]{JM} for a detailed proof).
\end{remark}

\begin{corollary}\label{corollary:Lp-strong}
Under the  assumptions of Theorem \ref{theorem:Lp}, Kolmogorov's inequality yields the following strong-type inequalities: for any $r\in[p_0,q)$ we have
\begin{equation}\label{Kolmo-conclusion}
{\|B_Q f\|}_{L^{r}, Q} \lesssim \sum_{k=2} ^\infty \gamma_k\,a(2^k\,Q)
\end{equation}
with some fast decay coefficients $(\gamma_k)_k$.
\end{corollary}

In some applications one cannot check that $\tilde{a}\in D_{q}$: the overlap of the cubes $2^k\,Q_i$ introduces some growing coefficients that lead to show that one has a $D_q$ condition with a new functional $\bar{a}$ on the right hand side, where
$$
\bar{a}(Q)
=
\sum_{k=2} ^\infty \bar{\gamma}_k\,a(2^k\,Q),
$$
and $\bar{\gamma}_k\gg \tilde{\gamma}_k$. In this way, one can modify the proof and obtain a similar conclusion with a worse sequence on the right hand side. The following result contains this as a particular case:

\begin{theorem}\label{theorem:Lp:two-functionals}
Under the assumptions of Theorem \ref{theorem:Lp} assume that there is another functional $\bar{a}$ such that the pair $(\tilde{a},\bar{a})\in D_q$: there exists $C_{(\tilde{a},\bar{a})}\ge 1$ such that for each cube $Q$ and any family $\{Q_i\}_i\subset Q$ of pairwise disjoint cubes,
\begin{equation}\label{eq:DDq}
\sum_{i} \tilde{a}(Q_i)^q\, |Q_i|\leq C_{(\tilde{a},\bar{a})}^q\, \bar{a}(Q)^q\, |Q|.
\end{equation}
Then for every cube $Q$,
\begin{equation}\label{conclusion:theor-Lp:two-functionals}
\|B_Q f\|_{L^{q,\infty}, Q}
\lesssim
\bar{a}(2\,Q).
\end{equation}
\end{theorem}

\subsection{Some improvements}

\subsubsection{About the commutative assumption} \label{subsec:comm}

In Definition \ref{def:Off} we require that the operators $A_Q$ (or $B_Q$) commute. From the proof is follows that the commutativity is only used in Proposition \ref{prop:II}. More precisely,
it is easy to see (details are left to the interested reader) that it suffices to have that  $B_R$ commutes with $B_R A_Q$ and $B_RA_Q^2$ for all cubes $R\subset Q$. Then $(a)$ in Definition \ref{def:Off} could be replaced by this weaker (but less useful) condition and one still gets the same conclusion.

\medskip

Another possible way to avoid $(a)$ is to assume that $A_{R}A_Qf(x)=A_Qf(x)$ for $x\in 2\,R$ and for every $R\subset Q$. In that case, we do not need to assume neither $(a)$ nor $(d)$. Furthermore, in that situation the proof and formulation can be simplified in such a way that we can work with $B_Q$ in place of $B_Q^2$ ---in passing we notice that we have needed to carry out our proof with $B_Q^2$ precisely because of the first term in \eqref{eqn:Bqi-BQ:square} which is going to vanish in the proof of the following result.

\begin{theorem}\label{theorem:Lp-alternative}
Fix $1\le p_0< q<q_0 \le \infty$ and let us assume that $\mathbb{B}:=(B_Q)_{Q\in \Qcal}$  satisfies $(b)$ and $(c)$ in Definition \ref{def:Off}. Assume further than
\begin{equation}\label{replace-comm}
A_{R}A_Qf(x)=A_Qf(x),
\qquad x\in 2\,R,
\end{equation}
for all $R\subset Q$.
Given $f\in \mathcal{F}$ and a functional $a$, let us assume that \eqref{hyp:BQ:2k-Q} holds. If $a\in D_q$ then for every cube $Q$, we have
\begin{equation}\label{conclusion:theor-Lp-alternative}
\|B_Q f\|_{L^{q,\infty}, Q}
\lesssim
\sum_{k\ge 1} \eta_k\,a(2^k\,Q)
\end{equation}
with $\eta_1=1$ and $\eta_k=\alpha_k\,2^{k\,n/p_0}$ for $k\ge 2$.
\end{theorem}

Next we present a version of the previous result in which the operators $A_Q$ are local and therefore in \eqref{hyp:BQ:2k-Q} it suffices to assume just the case $k=1$. As we see below this is well suited for the classical oscillations operators.

\begin{theorem}\label{theorem:Lp-alternative-bis}
Fix $1\le p_0< q<q_0 \le \infty$ and let us assume that $\mathbb{B}:=(B_Q)_{Q\in \Qcal}$  satisfies $(b)$ and $(c)$ in Definition \ref{def:Off}. Assume further than \eqref{replace-comm} holds and that the operators $A_Q$ are ``local'' in the following sense: for any cube $Q$,
\begin{equation}\label{eqn:localization:AQ}
A_Q f= \chi_{2Q} A_Q (f \chi_{2Q}).
\end{equation}
Given $f\in \mathcal{F}$ and a functional $a$, let us assume that for every cube $Q$ the following holds:
\begin{equation} \label{hyp:BQ:2k-Q-bis}
\|B_Q f\|_{L^{p_0},2Q} \leq a(2Q).
\end{equation}
If $a\in D_q$, then for every cube $Q$, we have
\begin{equation}\label{conclusion:theor-Lp-alternative-bis}
\|B_Q f\|_{L^{q,\infty}, Q}
\lesssim a(2Q).
\end{equation}
\end{theorem}

\begin{remark}\label{remark:local-case}
Let us observe that by the localization property of $A_Q$ \eqref{eqn:localization:AQ}, \eqref{AQ:off:p-q} holds trivially since the left hand side vanishes for every $j\ge 1$. Thus,
$(c)$ in Definition \ref{def:Off} reduces to the on-diagonal term \eqref{AQ:on:p-q} (with $2\,Q$ in place of $4\,Q$ in the right hand side).
\end{remark}

\begin{corollary}\label{corol:classical-Poincare}
Given $f\in L^1_{\rm loc}(\re^n)$ and a functional $a\in D_q$, $1<q<\infty$, assume that for every cube $Q$, we have
\begin{equation}\label{hyp:local-case}
\aver{Q}|f-f_Q|
\le
a(Q).
\end{equation}
Then for every cube $Q$,
$$
\|f-f_Q\|_{L^{q,\infty},Q}
\le
C\,a(2\,Q).
$$

\end{corollary}

This result should be compared with those in \cite{Franchi-Perez-Wheeden}. In the general case of the spaces of homogeneous type \cite{Franchi-Perez-Wheeden} obtains $a(C\,Q)$ in the right hand side for some dimensional constant $C>1$. However, in the Euclidean setting with the $\infty$-distance (i.e., where the balls are the cubes) the right hand side is improved to $Q$. This is because they work with the localized dyadic Hardy-Littlewood maximal function and use the corresponding Calder\'on-Zygmund decomposition. Our proof uses the regular Hardy-Littlewood maximal function and the Whitney decomposition, since this is better adapted to operators that are not localized, as those that satisfy the off-diagonal decay in Definition \ref{def:Off}. This explains why we obtain $2\,Q$.

\begin{proof}
Given a cube $Q$ we set
$$
A_Q f(x)=f_Q\,\chi_{2\,Q}(x),
\quad
B_Qf(x)=f(x)-A_Qf(x)=f(x)-f_Q\,\chi_{2\,Q}(x),
$$
where we write $f_Q$ to denote the average of $f$ on $Q$. We take $p_0=1$, $q_0=\infty$. We are going to show that all the required hypotheses in Theorem \ref{theorem:Lp-alternative-bis} hold and then our desired estimate follows from that result.

Note that we trivially obtain $(b)$ in Definition \ref{def:Off}:
$$
\|B_Q f\|_{L^1(\re^n)}
\le
\|f\|_{L^1(\re^n)}+|f_Q|\,|2\,Q|
\le
(1+2^n)\,\|f\|_{L^1(\re^n)}.
$$
On the other hand it is easy to see that $(c)$ also holds, by Remark \ref{remark:local-case} it suffices to show \eqref{AQ:on:p-q}:
$$
\|A_Q (f\,\chi_{4\,Q})\|_{L^\infty(2\,Q)}
=
|f_Q|
\le
\aver{Q}\,|f|\,dx.
$$
Regarding \eqref{replace-comm}, for all $R\subset Q$ and $x\in 2\,R$,
$$
A_{R}A_Qf(x)
=
f_Q\,A_R(\chi_{2Q})(x)
=
f_Q\,(\chi_{2Q})_R
=
f_Q
=
A_Q f(x).
$$
Note also that, by definition, \eqref{eqn:localization:AQ} holds. Finally we see that  \eqref{hyp:local-case} implies \eqref{hyp:BQ:2k-Q-bis}:
\begin{align*}
\aver{2\,Q} |B_Q f(x)|\,dx
& =
\aver{2\,Q} |f(x)-f_Q|\,dx
\le
\aver{2\,Q} |f(x)-f_{2\,Q}|\,dx
+|f_Q-f_{2\,Q}|
\\
& \le
a(2\,Q)+
2^n\,\aver{2Q} |f(x)-f_{2\,Q}|\,dx
\le
(1+2^n)\,a(2\,Q).
\end{align*}
\end{proof}

\subsubsection{Weighted estimates}\label{section:weights}

A weight $w$ is a non-negative locally integrable function.
For any measurable set $E$, we write $w(E)= \int_E w(x)\,dx$ and
$$
\aver{Q} f\,dw
=
\aver{Q} f(x)\,dw(x)
=
\frac1{w(Q)}\,\int_{Q} f(x)\,w(x)\,dx.
$$
We use the following notation:
$$
\|f\|_{L^p(w),Q}=\|f\|_{L^p (Q, \frac{w\,dx}{w(Q)})}
\quad \textrm{and}
\quad
\|f\|_{L^{p,\infty}(w),Q}=\|f\|_{L^{p,\infty} (Q, \frac{w\,dx}{w(Q)})}.
$$

Let us recall the definition  of the Muckenhoupt classes of weights.
We say that a weight $w\in A_p$, $1<p<\infty$, if there exists a positive constant $C$ such that for every cube $Q$,
$$
\bigg(\aver{Q} w\,dx\bigg)\,
\bigg(\aver{Q} w^{1-p'}\,dx\bigg)^{p-1}\le C.
$$
For $p=1$, we say that $w\in A_1$ if there is a positive constant $C$ such
that for every cube $Q$,
$$
\aver{Q} w\,dx
\le
C\, w(y),
\qquad \mbox{for a.e. }y\in Q.
$$
We write $A_\infty=\cup_{p\ge 1} A_p$.
We also need to introduce the reverse H\"older classes. A weight $w\in RH_p$, $1<p<\infty$, if there is a constant $C$ such that for every cube $Q$,
$$
\bigg( \aver{Q} w^p\,dx \bigg)^{1/p} \leq C \bigg(\aver{Q} w\,dx\bigg).
$$
It is well known that $A_\infty=\cup_{r>1} RH_r$. Thus, for $q=1$ it is understood that $RH_1=A_\infty$.
Notice also that if $w\in RH_p$ then H\"older's inequality yields that for any cube $Q$ and for any measurable set $E\subset Q$, we have
\begin{equation}\label{eq:RH-sets}
\frac{w(E)}{w(Q)} \leq C\, \left(\frac{|E|}{|Q|}\right)^{1/p'}.
\end{equation}

Theorem \ref{theorem:Lp} can be extended to spaces with $A_\infty$ weights
as follows:

\begin{theorem}\label{theorem:Lp-w}
Fix $1\le p_0< q<q_0 \le \infty$ and $\mathbb{B}:=(B_Q)_{Q\in \Qcal}$ in $O(p_0,q_0)$.
Given $f\in \mathcal{F}$ and a functional $a$, let us assume that \eqref{hyp:BQ:2k-Q} holds. Given $w\in RH_{(q_0/q)'}$,
there exist fast decay coefficients $\{\tilde{\gamma}_k\}_{k\ge 1}$  such that if we define
\begin{equation} \label{eq:tilde-w}
\tilde{a}(Q)
=
\sum_{k=1} ^\infty \tilde{\gamma}_k\,a(2^k\,Q)
\end{equation}
and $\tilde{a}\in D_q(w)$, then for every cube $Q$,
\begin{equation}\label{conclusion:theor-Lp-w}
\|B_Q f\|_{L^{q,\infty}(w), Q}
\lesssim
\tilde{a}(2\,Q).
\end{equation}
\end{theorem}

Proceeding as in Corollary \ref{corollary:Lp-strong}, from this result we easily obtain weighted strong-type inequalities and in the left hand side inequality of \eqref{conclusion:theor-Lp-w} we can replace $L^{q,\infty}(w)$ by $L^r(w)$ for every $p_0\le r<q$. Besides,  as in Theorem \ref{theorem:Lp:two-functionals}, we can consider a weighted extension of the previous result, where we  assume that $(\tilde{a},\bar{a})\in D_q(w)$, and obtain the corresponding $L^{q,\infty}(w)$ estimate. Also, one can extend Theorems \ref{theorem:Lp-alternative} and \ref{theorem:Lp-alternative-bis} to this setting. The precise statements and proofs are left to interested reader.

We would like to emphasize that we start from the unweighted estimate \eqref{hyp:BQ:2k-Q} and conclude \eqref{conclusion:theor-Lp-w} which is a weighted estimate for the oscillation $B_Q f$.

For a particular version of these results with $p_0=1$, $q_0=\infty$ see \cite[Theorem 3.3]{JM}. Notice that in that case the assumption on $w$ reduces to $w\in A_\infty$.

\subsubsection{Exponential self-improvement}

As recalled in the introduction, the John-Niren\-berg inequality gives an exponential integrability of the oscillations for a function belonging to $BMO$. In \cite{MP} it was shown that if \eqref{des:Poincare generalizada} holds and $a\in D_\infty$ then $|f-f_Q|$ is exponentially integrable in $Q$. This was extended to the oscillations $|f-A_t f|$ in \cite{Jim} (see also \cite{DY-CPAM} for the $BMO$ case) assuming fast decay for the kernels.
In our setting we show that the exponential self-improvement follows in the case $q_0=\infty$ and $a\in D_\infty$ (this should be compared with \cite{Jim} where $p_0=1$, $q_0=\infty$ and $A_Q=A_{t_Q}$).

Before stating our result let us recall the definition of the localized and normalized Luxemburg norms associated to the space $\exp L$:
$$
\|f\|_{\exp L,Q}
=
\inf\left\{
\lambda>0:\aver{Q} \left(\exp\left(\frac{|f(x)|}{\lambda}\right)-1\right)\,dx\le 1
\right\}.
$$
Analogously, we define $\|f\|_{\exp L(w),Q}$ replacing $dx$ by $dw$.

\begin{theorem} \label{theor:exponential}
Fix $1\le p_0<\infty$ and $\mathbb{B}=(B_Q)_{Q\in \mathcal{Q}}$ in $O(p_0,\infty)$. Given $f\in \mathcal{F}$ and a functional $a$, let us assume that \eqref{hyp:BQ:2k-Q} holds. If $a\in D_\infty$ then for every cube $Q$ we have
\begin{equation}\label{conclusion:theor-expo}
\|B_Q f\|_{{\rm exp} L, Q}
\lesssim
\sum_{k\ge 1} \eta_k\,a(2^k\,Q),
\end{equation}
with $\eta_1=1$, $\eta_k=\alpha_k$, $k=2,3,4$ and $\eta_k=\max\{\alpha_k,\alpha_{k-3}\}$ for $k\ge 5$. Furthermore, for every $w\in A_\infty$ we have
\begin{equation}\label{conclusion:theor-expo:w}
\|B_Q f\|_{{\rm exp} L(w), Q}
\lesssim
\sum_{k\ge 1} \eta_k\,a(2^k\,Q).
\end{equation}
\end{theorem}

It trivially follows from the proof that we have a version of Theorems \ref{theorem:Lp-alternative} and \ref{theorem:Lp-alternative-bis} in the present context. The latter leads to an analogue of Corollary \ref{corol:classical-Poincare} that recovers \cite{MP}.

\section{Applications}\label{section:appl}

In this section we present some applications of our main result. First, we see that under some conditions assumed on the family of oscillations $\mathbb{B}=(B_Q)_{Q\in\mathcal{Q}}$ then the associated $BMO_p$ spaces are $p$-independent, and this can be seen as an abstract John-Nirenberg result.

In addition, we take into account that our main results are ruled by two different objects and their corresponding properties. Namely, we have the oscillation operators $\mathbb{B}:=(B_Q)_{Q\in \Qcal}$ for which we need to study the range where the off-diagonal properties in $O(p_0,q_0)$ hold, and we have the functionals $a$ for which we need to consider their membership to the summability classes $D_r$. Therefore, we consider these tasks separately in Sections \ref{section:functional} and \ref{section:oscillation}. Finally in Section \ref{section:applications-semi} we combine them to obtain self-improving results.

\subsection{Abstract John-Nirenberg inequalities}

Let $\mathbb{B}=(B_Q)_{Q\in\mathcal{Q}}\in O(p_0,q_0)$. For $p\in [p_0,q_0)$, we define the $L^p$-sharp maximal function ${\mathcal M}^\#_{\mathbb{B},p}$ on ${\mathcal F}$ as follows: for every function $f\in {\mathcal F}$
$$
{\mathcal M}^\#_{\mathbb{B},p} f (x)
:=
\sup_{x\in Q} \left( \aver{Q} |B_Qf|^p dx \right)^{1/p}.$$
Then we can define the $BMO_{\mathbb{B}}^p$ spaces related to the collection of oscillation operators ${\mathbb B}$ as
$$
BMO_{\mathbb{B}}^p:=\left\{ f\in {\mathcal F}:\  {\mathcal M}^\#_{\mathbb{B},p} f \in L^\infty \right\}$$
equipped with the corresponding seminorm
$$ \|f\|_{BMO_{\mathbb{B}}^p} := \left\| {\mathcal M}^\#_p f \right\|_{L^\infty}.$$
Clearly, for every $p<q$ we have ${\mathcal M}^\#_{\mathbb{B},p} f(x)\le {\mathcal M}^\#_{\mathbb{B},q} f(x)$ and thus  $\|f\|_{BMO_{\mathbb{B}}^p}\le \|f\|_{BMO_{\mathbb{B}}^q}$.

We have the following ``John-Nirenberg property'':

\begin{corollary} \label{corc2}
Let $1\le p_0<q_0\le \infty$ and assume that we are in one of the two following situations:
\begin{list}{$(\theenumi)$}{\usecounter{enumi}\leftmargin=.8cm
\labelwidth=.8cm\itemsep=0.2cm\topsep=.1cm
\renewcommand{\theenumi}{\roman{enumi}}}
\item The operators $B_Q=B_{\ell(Q)}$ only depend on the sidelength of the cube $Q$ and $\mathbb{B}=(B_Q)_{Q\in\mathcal{Q}}\in O(p_0,q_0)$.

\item $\mathbb{B}=(B_Q)_{Q\in\mathcal{Q}}$ satisfies $(b)$ and $(c)$ in Definition \ref{def:Off}, \eqref{replace-comm} and the ``localization'' property \eqref{eqn:localization:AQ}.
\end{list}
For every $p_0<p<q_0$ we have $\|f\|_{BMO_{\mathbb{B}}^{p_0}}\approx \|f\|_{BMO_{\mathbb{B}}^{p}}$ and consequently $BMO_{\mathbb{B}}^p=BMO_{\mathbb{B}}^{p_0}$.
Moreover, given $1\le p_0<p<s <q_0 \le \infty$, for every $f\in BMO_{\mathbb{B}}^{p_0}$, we have
\begin{equation}  \label{eq:JN2}
{\mathcal M}^\#_{\mathbb{B},p} f(x)   \lesssim M_s \big({\mathcal M}^\#_{\mathbb{B},p_0} f\big)(x),
\qquad \mbox{for a.e. }x\in \re^n.
\end{equation}
\end{corollary}

\begin{proof}
Assuming momentarily \eqref{eq:JN2}, we have that ${\mathcal M}^\#_{\mathbb{B},p} f(x)\lesssim \|f\|_{BMO_{\mathbb{B}}^{p_0}}$ and thus $\|f\|_{BMO_{\mathbb{B}}^p}\le \|f\|_{BMO_{\mathbb{B}}^{p_0}}$. The converse estimate is trivial as observed above, and therefore the first part of the statement is proved.

To show \eqref{eq:JN2}, let us define the following functional:
$$
a(Q):= \essinf_Q M_s\big({\mathcal M}^\#_{\mathbb{B},p_0} f\big),
$$
and notice that $a(Q)\le \|f\|_{BMO_{\mathbb{B}}^{p_0}}<\infty$. It is clear that  $a$ is doubling, indeed $a(2\,Q)\le a(Q)$.
We take $q$, $p<q<s$ and see that $a\in D_q$: given a family of pairwise disjoint cubes $\{Q_i\}_i\subset Q$, we have
\begin{multline*}
\sum_i a(Q_i)^q\,|Q_i|
=
\sum_i \essinf_{Q_i} \big(M_s \big({\mathcal M}^\#_{\mathbb{B},p_0} f\big)\big)^q\,|Q_i|
\le
\sum_i\,\int_{Q_i} M_s \big({\mathcal M}^\#_{\mathbb{B},p_0} f\big)(x)^q\,dx
\\
\le
\int_Q M \big( ({\mathcal M}^\#_{\mathbb{B},p_0} f)^s\big)(x)^{q/s}\,dx
\lesssim
\essinf_{x\in Q} M \big( ({\mathcal M}^\#_{\mathbb{B},p_0} f)^s\big)(x)^{q/s}\,|Q|
=
a(Q)^q\,|Q|,
\end{multline*}
where we have used that $q<s$ and the well-known fact that $\big(M \big( ({\mathcal M}^\#_{\mathbb{B},p_0} f)^s\big)\big)^{q/s}\in A_1$ since $M \big( ({\mathcal M}^\#_{\mathbb{B},p_0} f)^s\big)<\infty$ a.e.

By definition of the sharp maximal function we have
\begin{align*}
\left(\aver{Q} |B_Q f|^{p_0} \right)^{1/p_0}
& \leq \essinf_{x\in Q} {\mathcal M}^\#_{\mathbb{B},p_0} f (x)
\le
\left(\aver{Q} ({\mathcal M}^\#_{\mathbb{B},p_0}f)^s\,dx\right)^{1/s}\\
& \le
\essinf_{x\in Q} M_s \big({\mathcal M}^\#_{\mathbb{B},p_0} f\big)(x)
=
a(Q),
\end{align*}
and therefore \eqref{hyp:BQ:2k-Q} holds for $k=0$.

In situation $(i)$, where the operators $B_Q=B_{\ell(Q)}$ only depend on the sidelength of the cube, it follows that \eqref{hyp:BQ:2k-Q} holds for all $k\ge 1$ by Remark \ref{remark:side-Q} and
the fact that  $a\in D_{q} \subset D_{p_0}$ since $p_0<q$. Then can apply Theorem \ref{theorem:Lp} and Remark \ref{rem:main-doubling} to obtain $\|B_Q f\|_{L^{q,\infty},Q}\lesssim a(Q)$.

In the context of local operators $A_Q$, situation $(ii)$, by Theorem \ref{theorem:Lp-alternative-bis}, we need to show \eqref{hyp:BQ:2k-Q-bis} (i.e., \eqref{hyp:BQ:2k-Q} with $k=1$). Indeed, using \eqref{replace-comm}, we have
$$
B_Qf(x) = f(x)-A_Q(A_{2Q} f)(x) - A_Q(B_{2Q}f)(x) = B_{2Q}f(x) - A_Q(B_{2Q} f)(x),
\quad x\in 2\,Q.
$$
Hence, using the ``localization'' property \eqref{eqn:localization:AQ}, \eqref{AQ:on:p-q}, and proceeding as before we get
$$
\left(\aver{2\,Q} |B_Q f|^{p_0} \right)^{1/p_0}
\leq (1+2^{-n/p_0}\,\alpha_2) \left(\aver{2Q} |B_{2Q} f|^{p_0} \right)^{1/p_0} \lesssim a(2Q).
$$
Next, we apply Theorem \ref{theorem:Lp-alternative-bis} to obtain $\|B_Q f\|_{L^{q,\infty},Q}\lesssim a(2\,Q)\le a(Q)$.

In either scenario we have shown that $\|B_Q f\|_{L^{q,\infty},Q}\lesssim a(Q)$. This and Kolmogorov's inequality (since $p<q$) give that for every $Q\ni x$,
$$
\left( \aver{Q} |B_Qf|^p dx \right)^{1/p}
\le
C\,\essinf_Q M_s\big({\mathcal M}^\#_{\mathbb{B},p_0} f\big)
\le
C\,
M\big(M_s\big({\mathcal M}^\#_{\mathbb{B},p_0} f\big)\big)(x).
$$
Taking the supremum over all the cubes $Q\ni x$ and using that $M_s\big({\mathcal M}^\#_{\mathbb{B},p_0} f\big)\in A_1$ since $s>1$ we obtain as desired \eqref{eq:JN2}.
\end{proof}

\begin{remark}
In the recent paper \cite{BZ3}, John-Nirenberg inequalities for $BMO$ spaces are also under consideration using a ``Hardy space point of view''.  There the authors impose the condition that the operators $B_Q$ are bounded from $BMO_{p_0}$ to $BMO_{q_0}$, more precisely for all cubes $R\subset Q$
\begin{equation}\label{eq:BMO}
\left(\aver{2R} |B_R A_Q f|^{q_0}\,dx \right)^{1/q_0} \lesssim \|f\|_{BMO_{p_0}}.
\end{equation}
It is interesting to point out that \eqref{eq:BMO} is similar in nature to $(d)$ in Definition \ref{def:Off}, although the last one is weaker since the right hand side of \eqref{BR-AQ:off:p-q} is just an $L^{p_0}$-average in place of an $L^{p_0}$-oscillation. We also note that commutativity was not assumed in \cite{BZ3} although property \eqref{eq:BMO} reflects some kind of commutativity between the oscillations with a gain of integrability. These two different points of view, the one presented here where commutativity between the $B_Q$'s is required, and the one in \cite{BZ3} with stronger properties assumed on the oscillations, seem to be quite similar. It would be interesting to combine both methods and provided a unified approach.
\end{remark}

\subsection{Examples of functionals}\label{section:functional}

In this section we consider some functionals $a$ and study their membership to the classes $D_r$. The examples that we consider are taken from \cite{JM} (see also \cite{BJM}) and we refer the reader to that reference for full details. Let us notice that the names that are assigned to these examples will become clear in Section \ref{section:applications-semi}

\begin{example}[$BMO$ and Lipschitz]\label{ex:BMO}
We consider the functionals $a(Q)=|Q|^{\alpha/n}$ with $\alpha\ge 0$. For the classical oscillations $a$ is associated to $BMO$ for $\alpha=0$ and to the Lipschitz (Morrey-Campanato) spaces for $\alpha>0$. It is trivial that $a\in D_\infty$ and then $a\in D_r$ for all $1\le r<\infty$ and $a\in D_r(w)$ for all $1<r<\infty$ and $w\in A_\infty$. Note also that $a$ is clearly doubling. A more general example of a functional in $D_\infty$ is $a(Q)=\varphi(a(Q))$ with $\varphi$ non-decreasing and non-negative (see \cite{JM} for the motivation and more details).
\end{example}

\begin{example}[Fractional averages]\label{ex:fract}
These are related to the concept of higher gradient introduced by J. Heinonen and P. Koskela in
\cite{Heinonen-Koskela-1}, \cite{Heinonen-Koskela-2}. Given $\lambda\geq1$, $0<\alpha<n$, and a weight $u\in A_\infty$ (see Section \ref{section:weights} for the precise definition),
we set
$$
a(Q)= \ell(Q)^\alpha\, {\bigg(\frac{u(Q)}{|Q|}\bigg)}^{1/s}.
$$
If $s\ge n/\alpha$ then $a\in D_\infty$ and therefore $a\in D_r$ for all $1\le r<\infty$. Otherwise, $1\le s<n/\alpha$, in \cite{Franchi-Perez-Wheeden}, it is shown that $a\in D_r$ for $1<r<s\,n/(n-\alpha\, s)$ for general weights $u$. Assuming further $u\in A_\infty$  they also proved that $a\in D_{\frac{s\,n}{n-\alpha\,s}+\epsilon}$ for some $\epsilon>0$ (depending on $u$) and that $a$ is doubling.
\end{example}

\begin{example}[Reduced Poincar\'e inequalities]
Given $1\le p<\infty$ and $0\le h\in L^s_{\rm loc}(\re^n)$, $1\le s<\infty$ we take
$$
a(Q)=\ell(Q)\,\left(\aver{Q} h^s\, dx\right)^{1/s}.
$$
Setting $s^*:=s\,n/(n-s)$ if $1\le s<n$ and $s^*:=\infty$ if $s\ge n$ one can show that $a\in D_{s^*}$ (see \cite{Franchi-Perez-Wheeden}, \cite{JM}).

If we applied Theorem \ref{theorem:Lp} we would need to see that the corresponding expanded functional $\tilde{a}$ satisfy a $D_q$ condition. In doing that, it was shown in \cite{JM}, \cite{BJM} that one needs to change the sequence defining $\tilde{a}$. This leads to use Theorem \ref{theorem:Lp:two-functionals} (in place of Theorem \ref{theorem:Lp}) and therefore the applications that may arise from these functionals are essentially those contained in the following example about expanded Poincar\'e inequalities.
\end{example}

\begin{example}[Expanded Poincar\'e inequalities] \label{ex:EPI}
We consider an expansion of the functionals in the previous example:
$$
a(Q) = \sum_{k=0}^\infty \gamma_k\, \ell(2^k\,Q)\, \left(\aver{2^k\,Q} h^s\, dx\right)^{1/s}.
$$
Again $a\in D_\infty$ if $s\ge n$. For $1\le s<n$, when trying to obtain the  $D_q$ condition we need to take  into account the overlap of the dilated cubes ---the cubes $\{Q_i\}_i$ are disjoint but we have  to consider their dilations $\{2^k\,Q_i\}_i$ for all $k\ge 0$.  This leads to change the sequence $\gamma_k$ and so we end up with a $D_q$ condition for a pair of functionals (see \eqref{eq:DDq} above).

\begin{lemma}\label{lemma:D-q:EPI}
For every $1\le q<s^*$ we have that $(a,\bar{a})\in D_q$ where
$$
\bar{a}(Q)=\sum_{k\geq0}
\bar{\gamma}_k\, \ell(2^k\,Q)\, \left(\aver{2^k\,Q} h^s\, dx\right)^{1/s},
$$
with $\bar{\gamma}_0=C\,\gamma_0$ and $\bar{\gamma}_k=2^{-k\,n\,(\frac1s-\frac1q)^+}\,\sum_{l=k-1}^\infty \gamma_l\,2^{l\,n\,(\frac1s-\frac1{q})^+}$, $k\ge 1$.
\end{lemma}

Here we have set $t^+:=\max(t,0)$. This result (essentially obtained in \cite{JM}, see also \cite{BJM}), follows at once from Lemma \ref{lemma:D-q:EPI:w} with $w\equiv 1$ and $\theta=1$.

Given $w\in A_1$ we can also consider the functional
$$
a_w(Q)
=
\sum_{k=0}^\infty \gamma_k\, \ell(2^k\,Q)\, \left(\aver{2^k\,Q} h^s\, dw\right)^{1/s}.
$$
If $s\ge n$, $a_w \in D_\infty$ follows from \eqref{w-A1-Ainfty} below ---which in turn implies that the functional $\ell(Q)\,(\aver{Q} h^s\,dw)^{1/s}$ is quasi-increasing. For $1\le s<n$, we have the following extension of Lemma \ref{lemma:D-q:EPI}:

\begin{lemma}\label{lemma:D-q:EPI:w}
Given $w\in A_1$, there exists $\theta=\theta(w)$, $0<\theta\le 1$, such that for every $1\le q<s^*$, we have  $(a_w,\bar{a}_w)\in D_q(w)$ where
$$
\bar{a}_w(Q)=\sum_{k\geq0}
\bar{\gamma}_k\, \ell(2^k\,Q)\, \left(\aver{2^k\,Q} h^s\, dw\right)^{1/s},
$$
with $\bar{\gamma}_0=C\,\gamma_0$ and $\bar{\gamma}_k=2^{-k\,n\,(\frac{1-\theta}{s}+\theta(\frac1s-\frac1q)^+)}\,\sum_{l=k-1}^\infty \gamma_l\,2^{l\,n\,(\frac{1-\theta}{s}+\theta(\frac1s-\frac1q)^+)}$, $k\ge 1$.
\end{lemma}

Let us observe that $\theta$ is such that $w\in RH_{(1/\theta)'}$ (see \eqref{w-A1-Ainfty}) and therefore if $w\equiv 1$ we can take $\theta=1$.

\end{example}

\subsection{Oscillation operators} \label{section:oscillation}

As we have already observed in Corollary \ref{corol:classical-Poincare}, our main results can be applied to the classical oscillations associated to the averaging operators and thus we can recover some results from \cite{Franchi-Perez-Wheeden}. Furthermore, in this section we show that the abstract framework (described in Definition \ref{def:Off}) includes the case where the operators $A_Q$ come from a  semigroup satisfying off-diagonal estimates. As mentioned before, we work with cubes but the arguments presented below adapt easily to balls.

Given $1\le p_0\le q_0\le\infty$, we say that a family of linear operators $\{T_t\}_{t>0}$  satisfies $L^{p_0}(\re^n)-L^{q_0}(\re^n)$ off-diagonal estimates,  if
there exist $C, c>0$ such that for all closed sets $E,F$ and functions $f$ supported on $F$,
\begin{equation}\label{ass:offdiag}
\left(\int_E \left|T_t f\right|^{q_0}\,dx\right)^{1/q_0} \leq C\, t^{-\frac{n}{2}\left(\frac{1}{p_0}-\frac{1}{q_0}\right)} e^{-c\frac{d(E,F)^2}{t}} \left(\int_{F} |f|^{p_0} \,dx\right)^{1/p_0}.
\end{equation}
Note that for $q_0=\infty$ we replace the $L^{q_0}(\re^n)$ norm by the $L^{\infty}(\re^n)$ norm. The value of $c$ has no interest to us provided it remains
positive. Thus, we will freely use the same $c$ from line to line.

Let $L$ be an operator on $\R^n$ generating a  semigroup $e^{-tL}$. Given $N\ge 1$ and a cube $Q$, we set
$$
B_{Q,N}:=\left(\Id-e^{-\ell(Q)^2 L}\right)^{N} \qquad \textrm{and} \qquad A_{Q,N}:=\Id-B_{Q,N}.
$$

\begin{proposition}\label{prop:off-semi}
Let $1\le p_0<q_0\le \infty$. Given $L$ as before, if $\{e^{-t\,L}\}_{t\geq 0}$ satisfies $L^{p_0}(\re^n)-L^{q_0}(\re^n)$ off-diagonal estimates and is strongly continuous on $L^p(\re^n)$ for some $p\in[p_0,q_0]$, $p<\infty$, then for all $N\ge 1$ the family $\mathbb{B}_N=\{B_{Q,N}\}_{Q\in\mathcal{Q}}$ satisfies $(a)$, $(b)$, $(c)$ in Definition \ref{def:Off}  with $\alpha_k=C\,e^{-c\,4^k}$.
Moreover, if $\{(t\,L)^k\,e^{-t\,L}\}_{t>0}$ satisfies $L^{p_0}(\re^n)-L^{q_0}(\re^n)$ off-diagonal estimates for $0\le k\le N$ with $N\ge  n/(2\,q_0)$, then $(d)$ holds with $\beta_k=C\,e^{-c\,4^k}$.  Consequently, $\mathbb{B}_N$ is $O(p_0,q_0)$.
\end{proposition}

\begin{remark}\label{rem:continuity}
By \cite[Proposition 4.4]{AM2} and since $\{e^{-t\,L}\}_{t\geq 0}$ satisfies $L^{p_0}(\re^n)-L^{q_0}(\re^n)$ off-diagonal estimates, we have that if there exists $p\in[p_0,q_0]$, $p<\infty$, such that $\{e^{-t\,L}\}_{t\geq 0}$ is strongly continuous on $L^p(\re^n)$, then it is strongly continuous on $L^r(\re^n)$ for all $r\in[p_0,q_0]$, $r<\infty$. In particular, for every function $f\in L^r(\re^n)$, $e^{-tL}f \longrightarrow f$ in $L^r(\re^n)$ as $t\to 0^+$.
\end{remark}

\begin{remark} \label{rem:}
Let us note that in place of the semigroup, we can produce similar arguments with more general operator $L$ ``of order $m$'' by considering
$$ B_{Q,N}:=\left(\Id-e^{-\ell(Q)^m L}\right)^{N}$$ or
the resolvants $ A_{Q,N}:=\left(\Id+\ell(Q)^mL\right)^{-N},$ provided they satisfy the previous off-diagonal estimates.
\end{remark}

\subsection{Second order divergence form elliptic  operators}\label{section:applications-semi} \label{subsec:appli}

We refer the reader to \cite{Aus} for the particular case of second order divergence form elliptic  operators $L$:
$$
Lf(x)= -\div( A(x) \nabla f(x))$$
with a bounded $n \times n$ matrix valued function $ A: \R^n \longrightarrow {\mathcal M}_{n}({\mathbb C}) $ satisfying the following ellipticity condition: there exist two constants $\Lambda\geq \lambda>0$ such that
$$
\lambda\,|\xi|^2
\le
{\rm Re}\, A(x)\,\xi\cdot\bar{\xi}
\quad\qquad\mbox{and}\qquad\quad
|A(x)\,\xi\cdot \bar{\zeta}|
\le
\Lambda\,|\xi|\,|\zeta|,
$$
for all $\xi,\zeta\in\co^n$ and almost every $x\in \re^n$. We have
used the notation
$\xi\cdot\zeta=\xi_1\,\zeta_1+\cdots+\xi_n\,\zeta_n$ and therefore
$\xi\cdot\bar{\zeta}$ is the usual inner product in $\co^n$.

The operator $-L$ generates a $C^0$-semigroup $\{e^{-t\,L}\}_{t>0}$ of contractions on $L^2(\re^n)$. There exist $p_-$, $p_+$, with $1\le p_-<2<p_+\le\infty$, such that $\{e^{-t\,L}\}_{t>0}$ ---and also $\{(t\,L)^k\,e^{-t\,L}\}_{t>0}$, $k\ge 1$--- satisfies $L^{p_0}(\re^n)-L^{q_0}(\re^n)$ off-diagonal estimates for all $p_-<p_0\le q_0<p_+$. It is also strongly continuous on $L^2(\re^n)$. Consequently, Proposition \ref{prop:off-semi} applies and this allows us to obtain the following result from our main result Theorem \ref{theorem:Lp} (and also from Corollary \ref{corollary:Lp-strong}):

\begin{corollary}\label{corol:semigroup-selfimpro}
Let $L$ be a second order divergence form elliptic  operator as above. Fix $N>n/(2\,p_+)$ and $p$, $q$ such that $p_-<p<q<p_+$. Given $f\in\mathcal{F}$ and a functional $a$, we assume that for every $k\ge 0$ and every cube $Q$,
\begin{equation}\label{hyp:BQ:2k-Q-semi}
\bigg(\aver{2^k\,Q} \big|\big(\Id-e^{-\ell(Q)^2 L }\big)^{N}f(x)\big|^p\,dx\bigg)^{1/p}
\le a(2^k\,Q).
\end{equation}
If we define $\tilde{a}$ as in \eqref{eq:tilde} with $\tilde{\gamma}_k=C\,e^{-c\,4^k}$ \textup{(}cf. Remark \ref{remark:sequences}\textup{)} and $\tilde{a}\in D_q$, we have that
$$
\big\| \big(\Id-e^{-\ell(Q)^2 L }\big)^{N}f\|_{L^{q,\infty},Q}
\lesssim \tilde{a}(2\,Q).
$$
Consequently, for every $p<r<q$,
$$
\bigg(\aver{Q} \big|\big(\Id-e^{-\ell(Q)^2 L }\big)^{N}f(x)\big|^r\,dx\bigg)^{1/r}
\lesssim \tilde{a}(2\,Q).
$$
\end{corollary}

\begin{remark}\label{remark:semigroup-selfimpro}
Notice that in a similar way we can obtain versions of Theorems \ref{theorem:Lp:two-functionals} and \ref{theorem:Lp-w}. We also observe that as mentioned in Remark \ref{remark:side-Q} if $a\in D_{p}$ then it suffices to assume \eqref{hyp:BQ:2k-Q-semi} only in the case $k=0$.
\end{remark}

Next, we are going to apply Corollary \ref{corol:semigroup-selfimpro} and the results mentioned in the previous remark to the functionals we have considered above.

\subsubsection{John-Nirenberg inequality for $BMO$ and Lipschitz spaces associated to $L$}

Given $\alpha\ge 0$, and $N>n/(2\,p_+)$ we take $\mathcal{F}:=\mathbb{M}_{\alpha,L}^{N,*}$ where $\mathbb{M}_{\alpha,L}^{N,*}=\cap_{\epsilon>0}\big(\mathbb{M}_{\alpha,L}^{\epsilon,N}\big)^*$  and $\mathbb{M}_{\alpha,L}^{\epsilon,N}$ is the collection of $f\in L^2(\re^n)$ such that $f$ is in the range of $L^k$ in $L^2(\re^n)$, $k=1,\dots,N$ and
$$
\|f\|_{\mathbb{M}_{\alpha,L}^{\epsilon,N}}
=
\sup_{j\ge 0} 2^{j\,(\frac{n}2+\alpha+\epsilon)}\sum_{k=0}^N \|L^{-k} f\|_{L^2(S_j(Q_0))}<\infty.
$$
Here $Q_0$ is the unit cube centered at the origin, $S_0(Q_0)=Q_0$ and $S_j(Q_0)=2^j\,Q_0\setminus 2^{j-1}\,Q_0$, $j\ge 1$. It is shown in \cite{Hof-May}, \cite{Hof-May-Mc} that  $\big(\Id-e^{-\ell(Q)^2 L }\big)^{N}f$ is globally well defined in the sense of distributions and also that $\big(\Id-e^{-\ell(Q)^2 L }\big)^{N}f\in L^2_{\rm loc}(\re^n)$  for every $f\in\mathcal{F}$. This means that according to Remark \ref{remark:distributions} we can apply our main results, and in particular, Corollary \ref{corol:semigroup-selfimpro} holds with this family  $\mathcal{F}$.

Given $p_-<p<p_+$, we say that an element $f\in \F$ belongs to $\Lambda_L^{\alpha,p}$ if
$$
\|f\|_{\Lambda_L^{\alpha,p}}
:=
\sup_{Q\subset\re^n}|Q|^{-\alpha/n}\bigg(\aver{Q} \big|\big(\Id-e^{-\ell(Q)^2 L }\big)^{N}f(x)\big|^p\,dx\bigg)^{1/p}<\infty.
$$
Let us notice that when $\alpha=0$ we write as usual $\Lambda_L^{0,p}=BMO_L^p$. For the particular case $p_-=1$, $p_+=\infty$, \cite{DY-CPAM} and \cite{DDY} introduced these spaces obtaining the corresponding John-Nirenberg (i.e., exponential integrability) estimates, see \cite{Jim}, \cite{JM} and \cite{BJM} for a different approach. The general case with $p=2$ and $\alpha=0$ was studied in \cite{Hof-May}, in which case the corresponding space is denoted by $BMO_L$, and it was shown that $BMO_L^p=BMO_L$ for every $p_-<p<p_+$. This John-Nirenberg inequality is obtained using duality and the Carleson measure characterization of the space $BMO_L$. It was also shown that $BMO_L$ is the dual of the Hardy space $H^1_{L^*}$. For $\alpha>0$ and $p=2$, \cite{Hof-May-Mc} establishes that $\Lambda_L^{\alpha}=\Lambda_L^{\alpha,2}$ is the dual of the Hardy space $H^s_{L^*}$ for $0<s<1$ and $\alpha=n(1/s-1)$.

Using Corollary \ref{corol:semigroup-selfimpro} we are going to show that the spaces $\Lambda_L^{\alpha,p}$ are independent of $p$ and in fact that for every $p_-<p<p_+$ we have the following John-Nirenberg inequality
$$
\|f\|_{\Lambda_L^{\alpha,p}}
\approx \|f\|_{\Lambda_L^{\alpha}}.
$$
Therefore our techniques allow us to give a direct proof of the John-Nirenberg inequality that reproves the result in \cite{Hof-May} when $\alpha=0$ without using any characterization or duality properties of the space $BMO_L$. For $\alpha>0$, our John-Nirenberg inequality is new.

We can combine this with \cite[Theorem 3.52]{Hof-May-Mc}, describing the duality between Hardy spaces $H_L^p$ and Lipschitz spaces:
$$
(H_{L^*}^s)^* = \Lambda_L^{\alpha,p}, \qquad 0<s<1,\quad\alpha=n(1/s-1).
$$
As just  discussed, this relation does not depend on $p\in(p_-,p_+)$. By this way, it follows that the Hardy space $H_{L^*}^s$ can be built on $(H^s_{L^*}, p,\epsilon, N)$-molecules (as defined in \cite[(3.6)]{Hof-May-Mc}), as long as $p\in(p_+',p_-')$ ---notice that $p_+'=p_-(L^*)$ and $p_-'=p_+(L^*)$. Hence, the Hardy space $H_{L^*}^s$ does not depend on this exponent $p$ (as stated without proof in \cite{Hof-May-Mc}).

We explain how to obtain the claimed estimate from Corollary \ref{corol:semigroup-selfimpro}. Notice that it suffices to show that for every $p_-<p_0\le 2\le q_0<p_+$ we have
$$
\|f\|_{\Lambda_L^{\alpha,p_0}}
\le
\|f\|_{\Lambda_L^{\alpha,q_0}}
\le
C\,\|f\|_{\Lambda_L^{\alpha,p_0}}
.
$$
The first estimate is trivial by Jensen's inequality. In order to obtain the second estimate we fix $f\in \F$, and we can assume by homogeneity that $\|f\|_{\Lambda_L^{\alpha,p_0}}=1$. Next we take
$a(Q)=|Q|^{\alpha/n}$ and by Example \ref{ex:BMO} we know that $a\in D_\infty$ and consequently $a$ belongs to $D_r$ for every $1< r<\infty$. Also, since $a$ is doubling we have that $\tilde{a}\approx a$. By Remark \ref{remark:semigroup-selfimpro} it suffices to check that \eqref{hyp:BQ:2k-Q-semi} holds for $k=0$, and this is nothing but the fact that $\|f\|_{\Lambda_L^{\alpha,p_0}}=1$. Thus, Corollary \ref{corol:semigroup-selfimpro} gives right away the desired estimate. We would like to point out that, as explained above, $\F$ is a family of distributions such that  $\big(\Id-e^{-\ell(Q)^2 L }\big)^{N}f$ is globally well defined in the sense of distributions and belongs to $L^2_{\rm loc}(\re^n)\subset L^{p_0}_{\rm loc}(\re^n)$ since $p_0\le 2$. Therefore, Remark \ref{remark:distributions} applies.

\subsubsection{Fractional averages}
Given $u\in A_\infty$, $0<\alpha<n$, $p_-<p<p_+$, and $N>n/(2\,p_+)$, let $f\in L^{p}_{\rm loc}(\re^n)$ be such that, for every cube,
$$
\bigg(\aver{Q} \big|\big(\Id-e^{-\ell(Q)^2 L }\big)^{N}f(x)\big|^p\,dx\bigg)^{1/p}
\le
\ell(Q)^\alpha\, {\bigg(\frac{u(Q)}{|Q|}\bigg)}^{1/s}=:a(Q).
$$
One could also take $f\in\F$ as above but for simplicity we prefer to work with functions.
Note that $a$ is doubling since so is $u$, therefore $\tilde{a}\approx a$.
Then from Corollary \ref{corol:semigroup-selfimpro} and proceeding as before we can conclude the following:
$$
\bigg(\aver{Q} \big|\big(\Id-e^{-\ell(Q)^2 L }\big)^{N}f(x)\big|^q\,dx\bigg)^{1/q}
\lesssim
\ell(Q)^\alpha\, {\bigg(\frac{u(Q)}{|Q|}\bigg)}^{1/s},
$$
in any of the following situations
\begin{list}{$(\theenumi)$}{\usecounter{enumi}\leftmargin=.8cm
\labelwidth=.8cm\itemsep=0.2cm\topsep=.1cm
\renewcommand{\theenumi}{\alph{enumi}}}

\item $s\ge n/\alpha$, $p_-<p<q<p_+$;

\item $1\le s< n/\alpha$, $p_+\le \frac{s\,n}{n-\alpha\,s}$, $p_-<p<q<p_+$;

\item $1\le s< n/\alpha$,  $p_-<p<q\le \frac{s\,n}{n-\alpha\,s} <p_+$;

\end{list}

By Example \ref{ex:fract} we know that $a\in D_\infty$ in case $(a)$, consequently $a\in D_r$ for every $1< r<\infty$, and $a\in D_{\frac{s\,n}{n-\alpha\,s}+\epsilon}$ for some $\epsilon>0$ depending on $u$ in cases $(b)$, $(c)$. By Remark \ref{remark:semigroup-selfimpro} it suffices to check that \eqref{hyp:BQ:2k-Q-semi} holds for $k=0$, and this what we have assumed. Consequently, Corollary \ref{corol:semigroup-selfimpro} gives the desired estimates.

\subsubsection{Expanded Poincar\'e inequalities}

Given $p_-<p<p_+$, and $N>n/(2\,p_+)$, let $f\in L^{p}_{\rm loc}(\re^n)$ and let $0\le h\in L^s_{\rm loc}$ be such that, for every cube,
\begin{equation}\label{eqn:hypo-EPI}
\bigg(\aver{Q} \big|\big(\Id-e^{-\ell(Q)^2 L }\big)^{N}f\big|^p\,dx\bigg)^{1/p}
\le
\sum_{j=0}^\infty \eta_j\,
\ell(2^j\,Q)\,\left(\aver{2^j Q} h^s\, dx\right)^{1/s}=:a(Q),
\end{equation}
where $\{\eta_j\}_{j\ge 0}$ is a fast decay sequence (typical examples are $\eta_j=C\, 2^{-j\,\sigma}$ with $\sigma>0$ or $\eta_j=C\,e^{-c\,4^j}$).

If $s\ge n$ we have observed above that $a\in D_\infty$ and consequently $\tilde{a}\in D_\infty$, hence $a, \tilde{a}\in D_r$ for every $1< r<\infty$.
By Remark \ref{remark:semigroup-selfimpro} we obtain that \eqref{hyp:BQ:2k-Q-semi} holds for all $k\ge 0$. Then  Corollary \ref{corol:semigroup-selfimpro} gives for every $p_-<p<q<p_+$,
$$
\bigg(\aver{Q} \big|\big(\Id-e^{-\ell(Q)^2 L }\big)^{N}f\big|^q\,dx\bigg)^{1/q}
\le
\sum_{j=2}^\infty \tilde{\eta}_{j-1}\,\ell(2^{j}\,Q)\,\left(\aver{2^j Q} h^s\, dx\right)^{1/s},
$$
where $\tilde{\eta}_j=C\,\sum_{k=1}^{j}e^{-c\,4^k}\,\eta_{j-k}$, $j\ge 1$.

If $1\le s<n$, for simplicity, we assume that $\eta_1>0$ and also that $\eta_k$ is quasi-decreasing, i.e., $\eta_k\le C\, \eta_j$ for every $j\le k$. Notice that we can always replace $\eta_k$ by $\sup_{j\ge k} \eta_k$ and this new sequence is decreasing and satisfies $\eta_1>0$. This happens unless $a(Q)$ is only defined by the term $j=0$ which goes back to the reduced Poincar\'e inequalities which as shown above satisfies $a\in D_s$ and therefore the following lemma holds, details are left to the reader. We need the following auxiliary result which says that (modulo a multiplicative constant) \eqref{hyp:BQ:2k-Q-semi} follows from the case $k=0$ which is our assumption.
\begin{lemma}\label{lemma:Poincare-Hyp-k}
Given $a$ as above, let us assume that $1\le p\le s<n$. Then for every $k\ge 0$, we have
$$
\bigg(\aver{2^k\,Q} \big|\big(\Id-e^{-\ell(Q)^2 L }\big)^{N}f\big|^p\,dx\bigg)^{1/p}
\le
C\,a(2^k\,Q).
$$
\end{lemma}

Next using that $\tilde{\gamma}_k=C\,e^{-c\,4^k}$ and we can compute $\tilde{a}$:
$$
\tilde{a}(Q)
=
C\,
\sum_{j=1}^\infty e^{-c\,4^j}\,a(2^j\,Q)
=
\sum_{j=1}^\infty \tilde{\eta}_j\,
\ell(2^j\,Q)\,\left(\aver{2^j Q} h^s\, dx\right)^{1/s}
$$
where $\tilde{\eta}_j=C\,\sum_{k=1}^j e^{-c\,4^k}\,\eta_{j-k}$.
As observed above Lemma \ref{lemma:D-q:EPI} gives that $(\tilde{a},\bar{a})\in D_q$, for $1\le q<s^*$, where
$$
\bar{a}(Q)=\sum_{j=1}^\infty
\bar{\gamma}_j\, \ell(2^j\,Q)\, \left(\aver{2^j\,Q} h^s\, dx\right)^{1/s},
$$
and $\bar{\gamma}_j=2^{-j\,n\,(\frac1s-\frac1q)^+}\,\sum_{l=j-1}^\infty \tilde{\eta}_l\,2^{l\,n\,(\frac1s-\frac1{q})^+}$ (for $\bar{\gamma}_1$ we set $\tilde{\eta}_0=0$).
Using all these and Remark \ref{remark:semigroup-selfimpro} we obtain that \eqref{eqn:hypo-EPI} implies
\begin{equation}\label{eqn:concl-EPI}
\bigg(\aver{Q} \big|\big(\Id-e^{-\ell(Q)^2 L }\big)^{N}f\big|^q\,dx\bigg)^{1/q}
\lesssim
\bar{a}(2Q)
=
\sum_{j=2}^\infty
\bar{\gamma}_{j-1}\, \ell(2^j\,Q)\, \left(\aver{2^j\,Q} h^s\, dx\right)^{1/s},
\end{equation}
in any of the following situations:
\begin{list}{$(\theenumi)$}{\usecounter{enumi}\leftmargin=.8cm
\labelwidth=.8cm\itemsep=0.2cm\topsep=.1cm
\renewcommand{\theenumi}{\alph{enumi}}}

\item $p_-<p\le s< n$, $p_+\le s^*$, $p_-<p<q<p_+$;

\item $p_-<p\le s< n$, $p_-<p<q<s^*<p_+$;

\end{list}

For instance, if $\eta_j=C\,2^{-\sigma\,j}$ with $\sigma>0$, then it is easy to see that  $\tilde{\eta}_j\approx \eta_j$ for $j\ge 1$ and thus $\bar{\gamma}_j\approx 2^{-j\,\sigma}$ provided $\sigma>n(1/s-1/q)^+$. If $\eta_j=C\,e^{-c\,4^j}$ then $\tilde{\eta}_j\lesssim e^{-c'\,2^j}$ and then we can take $\bar{\gamma}_j\approx e^{-c''\,2^j}$.

As observed in \cite[Section 4.3]{JM} the previous estimates give some global pseudo-Poincar\'e inequalities. Indeed we can easily follow the computations there to obtain that our initial assumption  \eqref{eqn:hypo-EPI} (which yields \eqref{eqn:concl-EPI}) implies that, for all $t>0$,
$$
\big\|\big(\Id-e^{-t L }\big)^{N}f\big\|_{L^q(\re^n)}
\lesssim
t^{\frac{n}{2}\,(\frac1q-\frac1{s^*})}\,\|h\|_{L^s(\re^n)},
$$
provided $p$, $q$, $s$ satisfy either $(a)$ or $(b)$, $s\le q$ and $\{\bar{\gamma}_{j}\,2^j\}_{j\ge 1}\in\ell^1$.

Next, we see that \eqref{eqn:hypo-EPI} (which is an unweighted estimate) gives some generalized weighted Poincar\'e inequalities. Assume that \eqref{eqn:hypo-EPI} holds and take $w\in A_1$. Then we easily obtain
$$
\bigg(\aver{Q} \big|\big(\Id-e^{-\ell(Q)^2 L }\big)^{N}f\big|^p\,dx\bigg)^{1/p}
\lesssim
\sum_{j=0}^\infty \eta_j\,
\ell(2^j\,Q)\,\left(\aver{2^j Q} h^s\, dw\right)^{1/s}=:a_w(Q).
$$
Notice that we indeed have $a(Q)\lesssim a_w(Q)$.

As observed above, if $s\ge n$ then $a_w\in D_\infty$. Then proceeding as above we obtain that the weighted version of
 Corollary \ref{corol:semigroup-selfimpro} gives for every $p_-<p<q<p_+$ and every $w\in A_1\cap RH_{(p_+/q)'}$,
$$
\bigg(\aver{Q} \big|\big(\Id-e^{-\ell(Q)^2 L }\big)^{N}f\big|^q\,dw\bigg)^{1/q}
\le
\sum_{j=2}^\infty \tilde{\eta}_{j-1}\,\ell(2^{j}\,Q)\,\left(\aver{2^j Q} h^s\, dw\right)^{1/s},
$$
where $\tilde{\eta}_j=C\,\sum_{k=1}^{j}e^{-c\,4^k}\,\eta_{j-k}$, $j\ge 1$.

Consider next the case $1\le s\le n$ (here we also assume the previous conditions on the sequence $\{\eta_k\}_k$). Applying Lemma \ref{lemma:Poincare-Hyp-k} and the fact $a(Q)\lesssim a_w(Q)$ we conclude that \eqref{hyp:BQ:2k-Q-semi} holds for all $k\ge 0$. Then we can compute $\tilde{a}_w(Q)$ and use Lemma \ref{lemma:D-q:EPI:w} to obtain that $(\tilde{a}_w,\bar{a}_w)\in D_q(w)$ for $1\le q<s^*$. All these and the two-functional weighted version of Corollary \ref{corol:semigroup-selfimpro} give that \eqref{eqn:hypo-EPI} implies that in any of the following situations
\begin{list}{$(\theenumi)$}{\usecounter{enumi}\leftmargin=.8cm
\labelwidth=.8cm\itemsep=0.2cm\topsep=.1cm
\renewcommand{\theenumi}{\alph{enumi}}}

\item $p_-<p\le s< n$, $p_+\le s^*$, $p_-<p<q<p_+$;

\item $p_-<p\le s< n$, $p_-<p<q<s^*<p_+$;

\end{list}
if $w\in A_1\cap RH_{(p_+/q)'}$, then
$$
\bigg(\aver{Q} \big|\big(\Id-e^{-\ell(Q)^2 L }\big)^{N}f\big|^q\,dw\bigg)^{1/q}
\lesssim
\bar{a}(2Q)
=
\sum_{j=2}^\infty
\bar{\gamma}_{j-1}\, \ell(2^j\,Q)\, \left(\aver{2^j\,Q} h^s\, dw\right)^{1/s},
$$
with $\bar{\gamma}_j=2^{-j\,n\,(\frac{1-\theta}{s}+\theta(\frac1s-\frac1q)^+)}\,\sum_{l=j-1}^\infty \tilde{\eta}_l\,2^{l\,n\,(\frac{1-\theta}{s}+\theta(\frac1s-\frac1q)^+)}$  and $\tilde{\eta}_j=C\,\sum_{k=1}^j e^{-c\,4^k}\,\eta_{j-k}$, $j\ge 1$, $\tilde{\eta}_0=0$. We notice that as in \cite[Section 4.3]{JM} the obtained estimates also imply weighted global pseudo-Poincar\'e inequalities. In the present case in either  scenario $(a)$ or $(b)$, assuming that  $q=s$ and that the coefficients $\bar{\gamma}_j$ decay fast enough we can conclude that for all $t>0$
$$
\big\|\big(\Id-e^{-t L }\big)^{N}f\big\|_{L^s(w)}
\lesssim
t^{\frac{1}{2}}\,\|h\|_{L^s(w)}.
$$

Similar results for weights in $A_r$ can be obtained in the same fashion. The precise formulations and statements are left to the interest reader.

We conclude this section by observing how to obtain \eqref{eqn:hypo-EPI} with different choices of $h$.
Using that $e^{-t\,L}1\equiv 1$ and proceeding as in \cite[Section 4.2]{JM} one can easily see that the $(p,p)$ Poincar\'e inequality and the off-diagonal estimates of $e^{-t\,L}$ yield \eqref{eqn:hypo-EPI} with $h=|\nabla f|$, $s=p$, $\eta_j=e^{-c\,4^j}$ ---we would like to call the reader's attention to the fact that such estimates could be obtained in contexts where Poincar\'e inequalities are unknown or they do not hold using the ideas in \cite[Proposition 4.5]{BJM}. On the other hand, following the proof of \cite[Proposition 4.5 (a)]{BJM} one can easily obtain that \eqref{eqn:hypo-EPI} holds with $s=p$ (and then we can always take $s\ge p$), $h=\sqrt{L}f$ and $\eta_j=2^{-j\,(2\,N-n)}$.

\section{Proofs of the main results } \label{sec:proof}

In this section we prove our main results and also some of auxiliary results from the applications.

\subsection{Proof of Theorem \ref{theorem:Lp}}

We refer the reader to \cite{Jim}, \cite{JM}, \cite{BJM} for similar results in the case $p_0=1$ and $q_0=\infty$ and with $A_Q$ associated to a semigroup or a generalized approximation of the identity. Here we borrow some ideas from these references.

Let us first observe that without loss of generality we can assume that $q_0<\infty$: indeed we can replace $q_0$ by $\tilde{q}_0$, chosen in such a way that $q<\tilde{q}_0<\infty$, since we have that $O(p_0,q_0)$ implies $O(p_0,\tilde{q}_0)$.

Next we obtain version of \eqref{hyp:BQ:2k-Q} for $B_Q^2=B_Q B_Q$:

\begin{lemma}\label{lemma:BQ2-hyp}
Assuming \eqref{hyp:BQ:2k-Q} and $O(p_0,q_0)$, then, for every $j\ge 1$, we have
$$
\left(\aver{2^j Q} \left|B_Q^2 f\right|^{p_0} \right)^{1/p_0}
\leq
C\,a(2^{j+1} Q)+
\sum_{k\geq 2} \alpha_{k+j} a(2^{k+j} Q).
$$
\end{lemma}

\begin{proof}
We first notice that $O(p_0,q_0)$ implies $O(p_0,p_0)$. Then \eqref{AQ:off:p-p} and \eqref{hyp:BQ:2k-Q} give
\begin{align*}
\left(\aver{2^j Q} \left|A_Q B_Q f\right|^{p_0} \right)^{1/p_0} &
\le
C\,\left(\aver{2^{j+1} Q} |B_Qf|^{p_0} \right)^{1/p_0}+
\sum_{k\geq 2} \alpha_{k+j} \left(\aver{2^{k+j} Q} |B_Qf|^{p_0} \right)^{1/p_0}
\\
& \le
C\,a(2^{j+1}Q)+
\sum_{k\geq 2} \alpha_{k+j} a(2^{k+j} Q).
\end{align*}
This, the following algebraic formula
\begin{equation}
B_Q^2=
B_Q\,B_Q
=
B_Q- A_Q B_Q,  \label{eq:algebre}
\end{equation}
and \eqref{hyp:BQ:2k-Q} yield at once the desired estimate

\end{proof}

Next we show that \eqref{conclusion:theor-Lp} follows from the corresponding estimate for $B_Q^2$:

\begin{lemma} \label{lemma:BQ2-reduction}
Under the  assumptions of Theorem \ref{theorem:Lp}, suppose that
\begin{equation} \label{desired:BQ^2}
\|B_Q^2 f\|_{L^{q,\infty}, Q}
\lesssim
\sum_{k=1} ^\infty \tilde{\gamma}_k\,a(2^{k+1}\,Q),
\end{equation}
with $\tilde{\gamma}_k\gtrsim \alpha_{k+1}$ for $k\ge 1$. Then \eqref{conclusion:theor-Lp} follows.
\end{lemma}

\begin{proof}
We proceed as in the previous proof using that $O(p_0,q_0)$ implies $O(p_0,q)$ since $p_0\le q< q_0$. From \eqref{AQ:off:p-q:j=1} and \eqref{hyp:BQ:2k-Q} we obtain
\begin{align}
\left(\aver{Q} \left|A_QB_Q f\right|^{q} \right)^{1/q} &
\le
2^{n/q}
\left(\aver{2\,Q} \left|A_QB_Q f\right|^{q} \right)^{1/q}
\leq
2^{n/q}\sum_{k\geq 2} \alpha_k \left(\aver{2^{k} Q} |B_Qf|^{p_0} \right)^{1/p_0} \nonumber
\\
& \le
2^{n/q}\sum_{k\geq 2} \alpha_k a(2^{k} Q)
\lesssim
\sum_{k=1} ^\infty \tilde{\gamma}_k\,a(2^{k+1}\,Q).
\label{eq:AQBQ}
\end{align}
We conclude the proof by invoking \eqref{eq:algebre} which together with \eqref{desired:BQ^2} and the trivial estimate $\|\cdot\|_{L^{q,\infty},Q} \le \|\cdot\|_{L^{q},Q}$ yield the desired inequality.
\end{proof}

Let $\{\tilde{\gamma}_k\}_{k\ge 0}$ be a fast decay sequence to be chosen (see Remark \ref{remark:sequences}) and define $\tilde{a}$ by means of it. We are going to show that
\begin{equation}\label{desired:BQ2}
\|B_Q^2 f\|_{L^{q,\infty}, Q}
\lesssim
\tilde{a}(2\,Q)
=
\sum_{k=1} ^\infty \tilde{\gamma}_{k}\,a(2^{k+1}\,Q)
\end{equation}
Then Lemma \ref{lemma:BQ2-reduction} implies the desired estimate \eqref{conclusion:theor-Lp} provided $\tilde{\gamma}_{k}\gtrsim \alpha_{k+1}$ (see Remark \ref{remark:sequences}).

We fix a cube $Q$ and assume that $\tilde{a}(2\,Q)<\infty$,
otherwise there is nothing to prove. Let $G(x) = \big| B_Q^2f(x) \big|\,\chi_{2\,Q}(x)$. By the Lebesgue differentiation theorem,  it suffices to estimate
$\|M_{p_0} G\|_{L^{q,\infty}, Q}$ where $M_{p_0} G(x)=M(G^{p_0})(x)^{1/p_0}$. Hence, we study the level sets
$\Omega_t = \{x\in\R^n: M_{p_0} G(x)>t\}$, $t>0$.

We first estimate the $L^{p_0}$-norm of $G$ by using Lemma \ref{lemma:BQ2-hyp} with $j=1$:
\begin{equation}\label{G-Lp0}
\left(\aver{2\,Q} G^{p_0}\,dx\right)^{\frac1{p_0}}
=
\left(\aver{2\,Q} |B_Q ^2f(x)|^{p_0}\,dx\right)^{\frac1{p_0}}
\le
C\,a(4Q)+\sum_{k\geq 2} \alpha_{k+1} a(2^{k+1} Q)
\le
\tilde{a}(2\,Q).
\end{equation}
In particular, $G\in L^{p_0}$ because $\tilde{a}(2\,Q)<\infty$.
Thus, using that $M$ is of weak-type $(1,1)$, we obtain for some numerical constant $c_0$
\begin{equation}\label{Omega-t}
|\Omega_t|
\lesssim \frac{1}{t^{p_0}}\,\|G\|_{L^{p_0}}^{p_0}
< \frac{c_0^{p_0} \, \tilde{a}(2\,Q)^{p_0}}{t^{p_0}}\, |Q|.
\end{equation}
Next, let $s>1$ be large enough to be chosen. We claim that the following good-$\lambda$ inequality holds:
given $0<\lambda<1$, for all $t>0$,
\begin{equation}\label{good-lambda}
|\Omega_{s\,t}\cap Q| \leq c\,\left[\left(\frac{\lambda}{s}\right)^{p_0}+s^{-q_0}\right]\, |\Omega_t\cap Q| + c\,
\bigg(\frac{c_0\,\tilde{a}(2\,Q)}{\lambda\,t}\bigg)^{q}\,|Q|,
\end{equation}
where $c$ only depends on $n$ and $\|\tilde{a}\|_{D_{q_0}}$.

Assuming this momentarily,  we then proceed as follows. We fix $N>0$. The previous inequality implies
\begin{align*}
\sup_{0<t\leq N/s} t^{q}\, \frac{|\Omega_{s\,t}\cap Q|}{|Q|}
&
\leq
c\,  \left[\left(\frac{\lambda}{s}\right)^{p_0}+s^{-q_0}\right] \, \sup_{0<t\leq N/s}t^{q}\, \frac{|\Omega_t\cap Q|}{|Q|}
+ c\, \left(\frac{c_0\,\tilde{a}(2\,Q)}{\lambda}\right)^{q}
\\
&
\leq c\,\left[\left(\frac{\lambda}{s}\right)^{p_0}+s^{-q_0}\right] \, \sup_{0<t\leq N} t^{q}\, \frac{|\Omega_t\cap
Q|}{|Q|} + c\, \left(\frac{c_0\,\tilde{a}(2\,Q)}{\lambda}\right)^{q}.
\end{align*}
Therefore,
\begin{equation}\label{good-lambda:sup}
\sup_{0<t\leq N} t^{q}\, \frac{|\Omega_{t}\cap Q|}{|Q|} \leq c\,
\left[s^{q-p_0}\lambda^{p_0}+s^{q-q_0}\right] \, \sup_{0<t\leq N}t^{q}\,\frac{|\Omega_t\cap Q|}{|Q|} +
c\, s^{q}\, \left(\frac{c_0\, \tilde{a}(2\,Q)}{\lambda}\right)^{q}.
\end{equation}
Observe that
$$
\sup_{0<t\leq N} t^{q}\, \frac{|\Omega_{t}\cap Q|}{|Q|} \leq N^{q}
<\infty.
$$
We take $s$ large enough and then $\lambda$ small enough such that
$$
c\, \left[s^{q-p_0}\lambda^{p_0}+s^{q-q_0}\right] \leq \frac{1}{2}.
$$
Hence, we can hide the first term in the right side of \eqref{good-lambda:sup} and get
$$
\sup_{0<t\leq N} t^{q}\, \frac{|\Omega_{t}\cap Q|}{|Q|} \lesssim \tilde{a}(2\,Q)^{q},
$$
with an implicit constant independent of the cube $Q$ and on $N$.
Taking limits as $N\to\infty$, we conclude that
$$
\|M_{p_0} G\|_{L^{q_,\infty},Q}\lesssim \tilde{a}(2\,Q).
$$
This estimate and the Lebesgue differentiation theorem give \eqref{desired:BQ2} as observed at the beginning of the proof.

We now show \eqref{good-lambda}. We split the proof in two cases. When $t$ is large,
we shall use the Whitney covering lemma adapted to the cube $Q$ and some auxiliary results for the operators $B_Q$.
When $t$ is small, the estimate is straightforward. Indeed if $0<t\leq c_0\,\tilde{a}(2\,Q)$ and $0<\lambda<1$,
$$
|\Omega_{s\,t}\cap Q| \le |Q| <  \left(\frac{c_0\,\tilde{a}(2\,Q)}{\lambda\,t}\right)^{q}\, |Q|,
$$
and this clearly implies \eqref{good-lambda}.

Suppose now that $t> c_0\,\tilde{a}(2Q)$.
We first need to build a dyadic grid in $\R^n$ adapted to the fixed cube $Q$ that consists of translations and dilations of the classical dyadic  structure. As in Theorem 5.2 and Subsection 5.1.1 of \cite{JM}, let $\{Q_i^t\}_i$ be the family of Whitney cubes (associated to such dyadic grid scaled to $Q$) so that $\Omega_t=\cup_i Q_i^t$: such a collection exists since $\Omega_t$ is open and by \eqref{Omega-t} we have
$\Omega_t\subsetneq \re^n$. Moreover, as a consequence of \eqref{Omega-t} and $t>
c_0\,\tilde{a}(Q)$, we also have $|\Omega_t| < |Q|$, and consequently for all $i$
\begin{equation}\label{lado}
\ell(Q_i^t) <\ell(Q).
\end{equation}
Next we are going to estimate
$|\Omega_{s\,t}\cap Q|$.
First, using that the level sets are nested, we obtain that
$$
|\Omega_{s\,t}\cap Q| = |\Omega_{s\,t}\cap \Omega_{t}\cap Q| \leq
\sum_{i} \left| \left\{x\in Q_i^t\cap Q: M_{p_0} G(x)> s\,t \right\} \right|.
$$
From now on, we only consider those cubes $Q_i^t$ such that $Q_i^t \cap Q
\neq\emptyset$ and since the cubes $Q_i^t$ are dyadic with respect to the cube
$Q$, \eqref{lado} implies that $Q_i^t
\subset Q$. We first localize
$G$. If $x\in Q_i^t$, we have
$$
M_{p_0} G(x) \leq M_{p_0}  (G\,\chi_{2\, Q_i^t})(x) + M_{p_0}  (G\,\chi_{(2\, Q_i^t)^c})(x) \leq M_{p_0}  (G\,\chi_{2\, Q_i^t})(x) + 23^{n/p_0}\, t,
$$
since, by \cite[Lemma 5.3]{JM},  $ M_{p_0} (G\,\chi_{(2\, Q_i^t)^c})(x) \leq 23^{n/p_0} t$ for all $x\in Q_i^t$.
Therefore, if $s>2\cdot 23^{n/p_0}$,
\begin{align}
|\Omega_{s\,t}\cap Q| & \leq
\sum_{i: Q_i^t\subset Q} \left| \left\{x\in Q_i^t: M_{p_0}
(G\,\chi_{2\,Q_i^t})(x) > (s-23^{n/p_0})\,t \right\} \right| \nonumber
\\& \leq
\sum_{i: Q_i^t\subset Q} \left| \left\{x\in Q_i^t: M_{p_0}
(G\,\chi_{2\,Q_i^t})(x) > s\,t/2 \right\} \right|. \label{Medida:conj-nivel}
\end{align}
Note that although in the previous estimate $G$ is
localized to $2\, Q_i^t$,
this function still involves the cube $Q$ on its term
$B_Q^2f$. Using that
\begin{align*}
M_{p_0} (G\,\chi_{2\,Q_i^t})
&= M_{p_0} \big(|B_Q^2f|\,\chi_{2\,Q_i^t}\big)
 \leq M_{p_0} \big(|B_{Q_i^t}^2 f|\, \chi_{2\,Q_i^t}\big) + M_{p_0} \big(|B_{Q_i^t}^2f-B_Q^2f|\, \chi_{2\,Q_i^t}\big),
\end{align*}
we obtain
\begin{multline*}
|\Omega_{s\,t}\cap Q|
\leq
\sum_{i: Q_i^t\subset Q} \big| \big\{x\in Q_i^t: M_{p_0} \big(|B_{Q_i^t}^2f|\, \chi_{2\,Q_i^t}\big)(x)> s\,t/4\big\} \big|
\\
+
\sum_{i: Q_i^t\subset Q} \big| \big\{x\in Q_i^t: M_{p_0} \big(|B_{Q_i^t}^2f-B_Q^2f|\, \chi_{2\,Q_i^t}\big)(x)> s\,t/4\big\} \big|
=I+II.
\end{multline*}

The following auxiliary results, whose proofs are deferred until Section \ref{subsection:proofs-aux}, allow us to estimate $I$ and $II$.

\begin{proposition} \label{prop:I} The term $I$ can be estimated as follows:
$$
I \lesssim \left(\frac{\lambda}{s}\right)^{p_0} |\Omega_t \cap Q| + \left( \frac{\tilde{a}(2\,Q)}{\lambda t} \right)^{q} |Q|.
$$
\end{proposition}

\begin{proposition}\label{prop:II}
We have the following estimate:
$$
\left(\aver{2\,Q_i^t}  \big|B_{Q_i^t}^2f-B_Q^2f\big|^{q_0}\,dx\right)^{1/q_0}
\lesssim
\tilde{a}(2Q)+t.
$$
\end{proposition}

Using the strong-type $(q_0,q_0)$ of the maximal function $M_{p_0}$, the last proposition and the fact that  $\tilde{a}(2\,Q)\lesssim t$, we have
\begin{multline*}
II
\lesssim
\frac{1}{(st)^{q_0}} \sum_{i: Q_i^t\subset Q} \int_{2\,Q_i^t}  \big|B_{Q_i^t}^2f-B_Q^2f\big|^{q_0}\,dx\lesssim
\frac{1}{(st)^{q_0}} \sum_{i: Q_i^t\subset Q} \left(\tilde{a}(2\,Q)+t\right)^{q_0} |Q_i^t|
\\
\lesssim
s^{-q_0}\sum_{i: Q_i^t\subset Q} |Q_i^t|
\lesssim
s^{-q_0}|\Omega_t \cap Q|.
\end{multline*}
This together with Proposition \ref{prop:I} lead to
$$
|\Omega_{s\,t}\cap Q| \leq  c\left[\left(\frac{\lambda}{s}\right)^{p_0}+s^{-q_0}\right] |\Omega_t \cap Q| + c\left( \frac{\tilde{a}(2\,Q)}{\lambda t} \right)^{q} |Q|
$$
for some numerical constant $c$ (independent on $\lambda$ and $q$). This completes the proof of \eqref{good-lambda} in this second case $\tilde{a}(2\,Q)\lesssim t$.\qed

\subsubsection{Proofs of the auxiliary results}\label{subsection:proofs-aux}

We refer the reader to the proof of Theorem \ref{theorem:Lp} for the notation used in this section. We begin by proving Proposition \ref{prop:I}.

\begin{proof}[Proof of Proposition \ref{prop:I}]
We note that
$$
I
= \sum_{i: Q_i^t\subset Q} \big| \big\{x\in Q_i^t: M_{p_0} \big(|B_{Q_i^t}^2f|\, \chi_{2\,Q_i^t}\big)(x)> s\,t/4\big\}\big|
=
\sum_{i\in \Gamma_1} \dots+\sum_{i\in \Gamma_2}\dots := \Sigma_1+\Sigma_2, $$
where
$$
\Gamma_1:=
\Big\{i:\ Q_i^t\subset Q,\ \aver{2\,Q_i^t} |B_{Q_i^t}^2f|^{p_0}\,dx\leq (\lambda\,t)^{p_0}
\Big\}$$
and$$
\Gamma_{2}:=
\Big\{i:\ Q_i^t \subset Q,\ \aver{2\,Q_i^t} |B_{Q_i^t}^2f|^{p_0}\,dx>(\lambda\,t)^{p_0} \Big\}.
$$
For the term corresponding to $\Gamma_1$, we use that the maximal function is of weak-type $(1,1)$, the definition of the set $\Gamma_1$ and then the fact that the cubes $Q_i^t$ are pairwise disjoint and they are all contained in $\Omega_t \cap Q$:
$$
\Sigma_1
\lesssim
\frac{1}{(s\,t)^{p_0}}
\sum_{i\in \Gamma_1}
\int_{2\,Q_i^t}
|B_{Q_i^t}^2f|^{p_0}\, dx
\lesssim
\left(\frac{\lambda}{s}\right)^{p_0} \sum_{i\in \Gamma_1} |Q_i^t|
\lesssim
\left(\frac{\lambda}{s}\right)^{p_0} |\Omega_t \cap Q|.
$$
This corresponds to the first part of the desired inequality.
For the indices $i\in \Gamma_2$ we use Lemma \ref{lemma:BQ2-hyp}:
\begin{equation}
(\lambda\,t) < \left(\aver{2\,Q_i^t}|B_{Q_i^t}^2f|^{p_0}\,dx\right)^{1/p_0} \leq
C\,a(4\,Q_i^t)+\sum_{k=2}^\infty \alpha_{k+1}\, a(2^{k+1}\,Q_i^t)
\le
\tilde{a}(Q_i^t). \label{eq:atilde:Dq}
\end{equation}
Then, taking into account that  $\tilde{a}\in D_{q}$ and that $\{Q_i^t\}_i\subset Q\subset 2\,Q$ is a collection of pairwise disjoint cubes, we obtain
\begin{equation}\label{use-Dq}
\Sigma_2\leq \sum_{i\in \Gamma_2} |Q_i^t|
\leq
\sum_{i\in \Gamma_2} \left(\frac{\tilde{a}(Q_i^t)}{\lambda t}\right)^{q} |Q_i^t|
\lesssim
\left(\frac{\tilde{a}(2\,Q)}{\lambda t}\right)^{q} |Q|.
\end{equation}
\end{proof}

\begin{proof}[Proof of Proposition \ref{prop:II}]

We first observe that
\begin{equation}\label{eqn:Bqi-BQ:square}
B_{Q_i^t}^2 - B_{Q}^2 = B_{Q_i^t}^2(\Id-B_Q^2) - (\Id-B_{Q_i^t}^2)B_Q^2,
\end{equation}
and estimate each term separately, here we use $\Id$ to denote the identity operator.
Using the commutative property $(a)$ of Definition \ref{def:Off}, we have
$$
B_{Q_i^t}^2(\Id-B_Q^2)
=
(B_{Q_i^t} A_Q)(2\,\Id-A_Q) B_{Q_i^t}.
$$
This, \eqref{BR-AQ:off:p-q} and \eqref{AQ:off:p-p} yield
\begin{align*}
&\left( \aver{2Q_i^t} \big| B_{Q_i^t}^2(\Id-B_Q^2) f(x) \big|^{q_0} dx \right)^{1/q_0}
\\
&
\quad\leq
\sum_{k\geq 1} \beta_{k+1} \left(\aver{2^{k+1} Q} |(2\,\Id-A_Q) B_{Q_i^t}f (x)|^{p_0}\,dx \right)^{1/p_0}
\\
&
\quad \le
\sum_{k\geq 1} 2\,\beta_{k+1} \left(\aver{2^{k+1} Q} |B_{Q_i^t} f(x)|^{p_0}\,dx \right)^{1/p_0}
+
\sum_{k\geq 1} C\,\beta_{k+1} \left(\aver{2^{k+2} Q} |B_{Q_i^t} f(x)|^{p_0}\,dx \right)^{1/p_0}
\\
&
\hskip5cm+
\sum_{k\geq 1}\sum_{l\ge 2} \beta_{k+1}\,\alpha_{k+l+1} \left(\aver{2^{k+l+1} Q} |B_{Q_i^t} f(x)|^{p_0}\,dx \right)^{1/p_0}
\\
&
\quad
\lesssim
\sum_{k\geq 2} \tilde{\beta}_{k} \left(\aver{2^{k} Q} |B_{Q_i^t} f(x)|^{p_0}\,dx \right)^{1/p_0}.
\end{align*}
Here we have used that $\{\beta_k\}_{k\ge 1}\in \ell^1$ and we have also set
$\tilde{\beta}_2=\beta_2$, $\tilde{\beta}_3=\max\{\beta_2,\beta_3\}$, and $\tilde{\beta}_k=\max\{\beta_{k-1},\beta_k,\alpha_k\}$ for $k\ge 4$.
Next we pick $k_i$ such that
\begin{equation}\label{lado-i}
2^{k_i}\, \ell(Q_i^t) \leq \ell(Q) < 2^{k_i+1}\, \ell(Q_i^t).
\end{equation}
Thus, since $\ell(Q_i^t)<\ell(Q)$ and $Q_i^t$ is a dyadic sub-cube of $Q$, we obtain $\ell(Q_i^t) \leq \ell(Q)/2$ and
\begin{equation}\label{Q_it-Q}
k_i \geq 1, \qquad 2^{k_i}\, Q_i^t \subset 2\, Q \qquad\hbox{and}\qquad  Q\subset 2^{k_i+2}\, Q_i^t.
\end{equation}
All these and \eqref{hyp:BQ:2k-Q} give
$$
\left(\aver{2^{k} Q} |B_{Q_i^t} f(x)|^{p_0}\,dx \right)^{1/p_0}
\!\!\le
4^{n/p_0} \left(\aver{2^{k+k_i+2} Q_i^t} |B_{Q_i^t} f(x)|^{p_0}\,dx \right)^{1/p_0}
\!\!\le
4^{n/p_0} a(2^{k+k_i+2} Q_i^t).
$$
Using this estimate and the fact that $\tilde{a}\in D_{q}\subset D_0$, we conclude that
\begin{equation}\label{use-D0}
\left( \aver{2Q_i^t} \big| B_{Q_i^t}^2(\Id-B_Q^2) f(x) \big|^{q_0} dx \right)^{1/q_0}
\lesssim
\sum_{k\geq 2} \tilde{\beta}_k a(2^{k+k_i+2} Q_i^t)
\lesssim
\tilde{a}(2^{k_i} Q_i^t)
\lesssim
\tilde{a}(2\,Q).
\end{equation}

On the other hand,
\begin{multline*}
\left( \aver{2Q_i^t} \big| (\Id-B_{Q_i^t}^2)B_Q^2 f(x) \big|^{q_0} dx \right)^{1/q_0}
\le
\left( \aver{2Q_i^t} \big| (\Id-B_{Q_i^t}^2)\big(\chi_{2^{k_i} Q_i^t}B_Q^2 f)(x) \big|^{q_0} dx \right)^{1/q_0}
\\
+
\left( \aver{2Q_i^t} \big| (\Id-B_{Q_i^t}^2)\big(\chi_{(2^{k_i} Q_i^t)^c}B_Q^2 f)(x) \big|^{q_0} dx \right)^{1/q_0}
=I_1+I_2.
\end{multline*}
Notice that $\Id-B_{Q_i^t}^2=A_{Q_i^t}(2\,\Id-A_{Q_i^t})$. Using first  \eqref{AQ:off:p-q:j=1} and then \eqref{AQ:off:p-p} we obtain
\begin{align}
&\left( \aver{2Q_i^t} \big| (\Id-B_{Q_i^t}^2)h(x) \big|^{q_0} dx \right)^{1/q_0}
\le
\sum_{k\ge 2} \alpha_k \left( \aver{2^{k}\,Q_i^t} \big| (2\,\Id-A_{Q_i^t})h(x) \big|^{p_0} dx \right)^{1/p_0}
\nonumber
\\
&\qquad
\le
2\,\sum_{k\ge 2} \alpha_k \left( \aver{2^{k}\,Q_i^t} |h(x)|^{p_0} dx \right)^{1/p_0}
+
C\,\sum_{k\ge 2} \alpha_k \left( \aver{2^{k+1}\,Q_i^t} |h(x)|^{p_0} dx \right)^{1/p_0}
\nonumber
\\
&\qquad\qquad\qquad
+
\sum_{k\ge 2}\sum_{l\ge 2} \alpha_k\,\alpha_{l+k} \left( \aver{2^{l+k}\,Q_i^t} |h(x)|^{p_0} dx \right)^{1/p_0}
\nonumber
\\
&\qquad
\lesssim
\sum_{k\ge 2}\hat{\alpha}_k\left( \aver{2^{k}\,Q_i^t} |h(x)|^{p_0} dx \right)^{1/p_0}, \label{off:BQ-2}
\end{align}
where we have used that $\{\alpha_k\}_{k\ge 2}\in\ell^1$ and we have taken
$\hat{\alpha}_2=\alpha_2$ and $\hat{\alpha}_k=\max\{\alpha_{k-1},\alpha_k\}$ for $k\ge 3$.

We use \eqref{off:BQ-2} to estimate $I_1$. Notice that $2^{k_i}\, Q_i^t \subset 2\, Q$ implies that if $x\in 2^{k_i} Q_i^t$, then $|B_Q^2 f(x)|= G(x)$.  Hence,
\begin{align*}
I_1
& \le \sum_{k\ge 2}\hat{\alpha}_k\left( \aver{2^{k} Q_i^t} \big| \chi_{2^{k_i} Q_i^t}(x) B_Q^2 f(x) \big|^{p_0} dx \right)^{1/p_0}
\\
& \le
\sum_{k\ge 2}\hat{\alpha}_k\left( \aver{2^{k} Q_i^t} G(x)^{p_0} dx \right)^{1/p_0}
\lesssim
t\,\sum_{k\ge 2}\hat{\alpha}_k
\lesssim t.
\end{align*}
Here we have used that $\hat{\alpha}_k$ is a fast decay sequence and also that from the Whitney covering lemma it follows that $10\, Q_i^t \cap \Omega_t^c\neq\emptyset$ and therefore, for every $k\ge 0$,
\begin{equation}\label{Whitney-prop}
\left( \aver{2^{k} Q_i^t} G(x)^{p_0} dx \right)^{1/p_0} \le 10^{n/p_0} t,\end{equation}
see \cite[Theorem 5.2, Lemma 5.3]{JM}.
On the other hand, to estimate $I_2$ we use again \eqref{off:BQ-2}, that $2^{k}\,Q_i^t\subset 2^{k}\,Q$ and Lemma \ref{lemma:BQ2-hyp}:
\begin{align*}
I_2
&\le
\sum_{k\ge 2}\hat{\alpha}_k\left( \aver{2^{k} Q_i^t} \big| \chi_{(2^{k_i} Q_i^t)^c}(x) B_Q^2 f(x) \big|^{p_0} dx \right)^{1/p_0}
\\
&\le
\sum_{k> k_i}\hat{\alpha}_k\left( \aver{2^{k} Q_i^t} \big| B_Q^2 f(x) \big|^{p_0} dx \right)^{1/p_0}
\\
&\lesssim
\sum_{k> k_i}\hat{\alpha}_k2^{k_i\,n/p_0}\left( \aver{2^{k} Q} \big| B_Q^2 f(x) \big|^{p_0} dx \right)^{1/p_0}
\\
&\le
\sum_{k\ge 2}\hat{\alpha}_k2^{k\,n/p_0}a(2^{k+1} Q)+
\sum_{k\ge 2}\sum_{l\geq 2} \hat{\alpha}_k2^{k\,n/p_0}\,\alpha_{l+k}\, a(2^{l+k} Q)
\\
&\lesssim
\sum_{k\ge 3}\tilde{\alpha}_ka(2^{k} Q)
\\
&\le
\tilde{a}(2\,Q).
\end{align*}
where $\tilde{\alpha}_3=\hat{\alpha}_2=\alpha_2$ and
$$
\tilde{\alpha}_k=\max\{2^{k\,n/p_0}\,\hat{\alpha}_{k-1},\alpha_k\}=\max\{2^{k\,n/p_0}\,\alpha_{k-2},2^{k\,n/p_0}\,\alpha_{k-1},\alpha_k\}
$$
for $k\ge 4$. Here we have used that $\{2^{kn/p_0}\,\hat{\alpha}_k \}_{k\ge 2}\in\ell^1$, which in turn is equivalent to $\{2^{kn/p_0}\,\alpha_k\}_{k\ge 2}\in\ell_1$.

Gathering the obtained estimates the proof is complete.
\end{proof}

\begin{remark}\label{remark:distributions}
For some applications it may be interesting to extend the class of ``functions'' in our main results. We have assumed (for simplicity) that $\mathcal{F}$ is a given family of functions in $L^{p_0}_{\rm loc}(\re^n)$. However, the previous proof can be carried out with no change and thus Theorem \ref{theorem:Lp} remains valid if $\mathcal{F}$ is a family of distributions such that $B_Q f$ is globally well defined in the sense of distributions and $B_Q f\in L^{p_0}_{\rm loc}(\re^n)$ for every $f\in\mathcal{F}$. The same applies to Theorems \ref{theorem:Lp:two-functionals}, \ref{theorem:Lp-alternative}, \ref{theorem:Lp-alternative-bis} and \ref{theorem:Lp-w}. Details are left to the reader.
\end{remark}

\begin{remark}\label{remark:sequences}
From the proof we can easily see how $\tilde{\gamma}_k$ is chosen. The sequences $\{\alpha_k\}_{k\ge 1}$ and $\{\beta_k\}_{k\ge 1}$ are given in $O(p_0,q_0)$. We have used that $\{\alpha_k\}_{k\ge 2}$ and $\{\beta_k\}_{k\ge 2}$ are in $\ell^1$.
We have set (modulo multiplying constants)
$$
\tilde{\beta}_2=\beta_2,
\qquad\quad
\tilde{\beta}_3=\max\{\beta_2,\beta_3\},
\qquad\quad
\tilde{\alpha}_2=0,
\qquad\quad
\tilde{\alpha}_3=\alpha_2,
$$
and, for $k\ge 4$,
$$
\tilde{\beta}_k=\max\{\beta_{k-1},\beta_k,\alpha_k\},
\qquad\quad
\tilde{\alpha}_k=\max \left\{  2^{\frac{k\,n}{p_0}}\alpha_{k-2},\ 2^{\frac{k\,n}{p_0}}\alpha_{k-1},\alpha_k \right\}.
$$
Then we define $\tilde{a}$ using the coefficients $\tilde{\gamma}_k$ (details are left to the reader):
$$
\tilde{\gamma}_1\gtrsim 1,
\quad
\tilde{\gamma}_2\gtrsim \max\{1,\alpha_2, \alpha_3\},
\quad
\tilde{\gamma}_3\gtrsim \max\{\alpha_2, \alpha_3, \alpha_4\},
\quad
\tilde{\gamma}_4\gtrsim \max\{\alpha_3, \alpha_4, \alpha_5,\beta_2\},
$$
and for $k\ge 5$
$$
\tilde{\gamma}_k\gtrsim \max\{\alpha_{k-2}, 2^{k\,n/p_0}\,\alpha_{k-1}, 2^{k\,n/p_0}\,\alpha_{k}, \alpha_{k+1},\beta_{k-3},\beta_{k-2}\}.
$$
Note that this choice guarantees that $\tilde{\gamma}_k\gtrsim \alpha_{k+1}$.

Let us notice that in the applications below $\alpha_k\approx\beta_k$, both sequences are quasi-decreasing and also $ 2^{k\,n/p_0}\,\alpha_{k}\lesssim \alpha_{k-1}$. Then we have
$\tilde{\gamma}_1\gtrsim 1$, $\tilde{\gamma}_2\gtrsim \max\{1,\alpha_2\}$, $\tilde{\gamma}_k\gtrsim \alpha_2$ for $k=3,4$ and $\tilde{\gamma}_k\gtrsim \alpha_{k-3}$ for $k\ge 5$.
\end{remark}

\begin{remark}\label{remark:sequences:Dq}
Notice that in the proof we have used the condition $\tilde{a}\in D_q$ in \eqref{use-Dq}. However, in the last inequality of \eqref{eq:atilde:Dq} we can replace the last term by the smaller functional $\hat{a}(Q_i^t)$ defined by
$\hat{a}(Q)=a(4\,Q)+\sum_{k=2}^\infty \alpha_{k+1}\,a(2^{k+1}\,Q)$. In this way, we can replace the hypothesis $\tilde{a}\in D_q$ by $\hat{a}\in D_q$ and obtain the same inequality \eqref{use-Dq}
after using that $\hat{a}(2\,Q)\lesssim \tilde{a}(2\,Q)$. Let us observe in \eqref{use-D0} we have also used $a\in D_q$ in a very mild manner via the $D_0$ condition. Thus, we would need to add the hypothesis that the functional $\sum_{k\ge 0} \tilde{\beta}_k\, a(2^k\,Q)$ is in $D_0$. As observed this follows, for instance, if that functional satisfies some $D_r$  condition or more in particular if $a\in D_0$. In applications it could be easier to check these new hypotheses, since the coefficients $\alpha_k$ decay faster than $\tilde{\gamma}_k$.
\end{remark}

\subsection{Proof of Theorem \ref{theorem:Lp:two-functionals}}

The proof of this result (including the choice of $\tilde{\gamma}_k$ in Remark \ref{remark:sequences}) is very similar to the argument given above, so we just give the main changes. We first emphasize that \eqref{eq:DDq} implies
\begin{equation}\label{eq:aa}
\tilde{a}(Q)\lesssim \bar{a}(Q)
\end{equation}
for any cube $Q$ (we just take a family consisting only of the cube $Q$).

We follow the proof of Theorem \ref{theorem:Lp}. As done in Lemma \ref{lemma:BQ2-reduction} it suffices to obtain that
\begin{equation} \label{desired2:BQ^2}
\|B_Q^2 f\|_{L^{q,\infty}, Q} \lesssim \bar{a}(2Q),
\end{equation}
provided $\tilde{\gamma}_k\gtrsim \alpha_{k+1}$. Note that the last term in \eqref{eq:AQBQ} is $\tilde{a}(2\,Q)$ and \eqref{eq:aa} implies $\tilde{a}(2\,Q)\lesssim\bar{a}(2\,Q)$.
Then we proceed as in the previous proof, fix $Q$ and assume that the right hand side of \eqref{conclusion:theor-Lp:two-functionals} is finite, that is, $\bar{a}(2Q)<\infty$. As just observed $\tilde{a}(2Q)\le \bar{a}(2Q)<\infty$. We define $G$ and $\Omega_t$ as before and we observe that \eqref{G-Lp0} and \eqref{Omega-t} hold. We are going to obtain the following good-$\lambda$ inequality: given $0<\lambda<1$, for all $t>0$,
\begin{equation}\label{good-lambda-2-functionals}
|\Omega_{s\,t}\cap Q| \leq c\,\left[\left(\frac{\lambda}{s}\right)^{p_0}+s^{-q_0}\right]\, |\Omega_t\cap Q| + c\,
\bigg(\frac{c_0\,\bar{a}(2\,Q)}{\lambda\,t}\bigg)^{q}\,|Q|.
\end{equation}
With this in hand we can obtain that $\|M_{p_0}G\|_{L^{q,\infty},Q}\lesssim\bar{a}(2Q)$ which in turns yields \eqref{desired2:BQ^2}.

As before  \eqref{good-lambda-2-functionals} is trivial when $0<t\lesssim \bar{a}(2Q)$. Otherwise we repeat the previous steps and it suffices to obtain that, under the present hypotheses, Propositions \ref{prop:I} and \ref{prop:II} hold replacing $\tilde{a}$ by $\bar{a}$. Regarding Proposition \ref{prop:I}, the estimate for $\Sigma_1$ is the same.  For $\Sigma_2$ we still get \eqref{eq:atilde:Dq} which together with \eqref{eq:DDq} yield the following analogue of \eqref{use-Dq}:
$$
\Sigma_2\leq \sum_{i\in \Gamma_2} |Q_i^t|
\leq
\sum_{i\in \Gamma_2} \left(\frac{\tilde{a}(Q_i^t)}{\lambda t}\right)^{q} |Q_i^t|
\lesssim
\left(\frac{\bar{a}(2\,Q)}{\lambda t}\right)^{q} |Q|.
$$

Regarding Proposition \ref{prop:II} we note that in \eqref{use-D0} we have used that $\tilde{a}\in D_q\subset D_0$. In this case the same computations yield
$$
\left( \aver{2Q_i^t} \big| B_{Q_i^t}^2(\Id-B_Q^2) f(x) \big|^{q_0} dx \right)^{1/q_0}
\le
\sum_{k\geq 1} \tilde{\beta}_k a(2^{k+k_i+3} Q_i^t)
\le
\tilde{a}(2^{k_i} Q_i^t)
\lesssim
\bar{a}(2\,Q).
$$
In the last inequality we have used \eqref{eq:DDq} and the facts that $2^{k_i} Q_i^t\subset 2Q$ and $|2^{k_i} Q_i^t|\approx |2Q|$ by \eqref{Q_it-Q} and \eqref{lado-i}. The rest of the argument remains the same and at the end of the proof we use \eqref{eq:aa} to obtain $I_2\lesssim \tilde{a}(2Q)\lesssim \bar{a}(2Q)$.
\qed

\subsection{Proof of Theorem \ref{theorem:Lp-alternative}}
The proof follows the same scheme as before, the main difference is that we replace everywhere $B_Q^2$ by $B_Q$. We set $\tilde{a}(Q)$ equal to the right hand side of \eqref{conclusion:theor-Lp-alternative}
and assume that $\tilde{a}(Q)<\infty$. Let us set $G(x)=|B_Q f(x)|\,\chi_{2\,Q}(x)$ and define the corresponding $\Omega_t$. Then we get the following substitute of \eqref{G-Lp0}:
\begin{equation}\label{G-Lp0bis}
\left(\aver{2\,Q} G^{p_0}\,dx\right)^{1/p_0}
=
\left(\aver{2\,Q} |B_Q f(x)|^{p_0}\,dx\right)^{1/p_0}
\le
a(2Q)
\le\tilde{a}(Q)
.
\end{equation}
Note that this implies that $G\in L^{p_0}$ and also \eqref{Omega-t} with $\tilde{a}(Q)$ in place of $\tilde{a}(2\,Q)$. Then we show that for $s$ large enough and every $0<\lambda<1$ and $t>0$,
$$
|\Omega_{s\,t}\cap Q| \leq c\,\left[\left(\frac{\lambda}{s}\right)^{p_0}+s^{-q_0}\right]\, |\Omega_t\cap Q| + c\,
\bigg(\frac{c_0\,\tilde{a}(Q)}{\lambda\,t}\bigg)^{q}\,|Q|.
$$
This implies as before the desired estimate.

To obtain the good-$\lambda$ inequality we only  consider the case $t\gtrsim \tilde{a}(Q)$ (the other case is trivial).
The proof follows the same path (replacing $B_R^2$ by $B_R$) and then we have to estimate
$$
I:= \sum_{i: Q_i^t\subset Q} \big| \big\{x\in Q_i^t: M_{p_0} \big(|B_{Q_i^t}f|\, \chi_{2\,Q_i^t}\big)(x)> s\,t/4\big\}\big|
$$
and
$$
II:= \sum_{i: Q_i^t\subset Q} \big| \big\{x\in Q_i^t: M_{p_0} \big(|B_{Q_i^t}f-B_Q f|\, \chi_{2\,Q_i^t}\big)(x)> s\,t/4\big\}\big|
.
$$

For $I$ we proceed as in Proposition \ref{prop:I} (with $B_{Q_i}$ replacing $B_{Q_i}^2$). The estimate for $\Sigma_1$ is the same. For $\Sigma_2$ we have that if $i\in\Gamma_2$, then
\begin{equation}\label{estimate-comm-sigma2}
(\lambda\,t) < \left(\aver{2\,Q_i^t}|B_{Q_i^t}f|^{p_0}\,dx\right)^{1/p_0} \leq
a(2\,Q_i^t).
\end{equation}
As in \cite[Section 5.1]{JM} we have that the family $\{2\,Q_i^t\}_i\subset 2Q$ splits into $c_n$ (with $c_n\le 144^n$) families $\mathcal{E}_j$ of pairwise disjoint cubes.
Then we use that $a\in D_q$ in each family to conclude that
\begin{align*}
\Sigma_2 & \leq \sum_{i\in \Gamma_2} |Q_i^t|
\le
\sum_{j=1}^{c_n}\sum_{Q_i\in\mathcal{E}_j} \left(\frac{a(2Q_i^t)}{\lambda t}\right)^{q} |2\,Q_i^t|
\\
& \le
\|a\|_{D_q}^q\,c_n
\left(\frac{a(2\,Q)}{\lambda t}\right)^{q} |2\,Q|
\lesssim
\left(\frac{\tilde{a}(Q)}{\lambda t}\right)^{q} |Q|.
\end{align*}

The main changes come into the estimate of $II$. We adapt the proof of Proposition  \ref{prop:II} as follows: note that for $x\in 2\,Q_i^t$ by \eqref{replace-comm} and since $Q_i^t\subset Q$ we obtain
$$
B_{Q_i^t}(\Id-B_Q)f(x)
=(\Id-A_{Q_i^t})A_Qf(x)
=
A_Qf(x)-A_{Q_i^t}A_Qf(x)
=
0
$$
Then, for every $x\in 2\,Q_i^t$,
\begin{equation}\label{eqn:commu-reduction}
B_{Q_i^t}f(x) - B_{Q}f(x) = B_{Q_i^t}(\Id-B_Q)f(x) - (\Id-B_{Q_i^t})B_Qf(x)
=
-A_{Q_i^t}\,B_Qf(x).
\end{equation}
Let us observe that the term that has disappeared corresponds to the first term in \eqref{eqn:Bqi-BQ:square}, and to estimate that quantity we used $(a)$ and $(d)$ in Definition \ref{def:Off},

Next we pick $k_i$ as before and use \eqref{AQ:off:p-q:j=1} (which follows from $(c)$) to obtain  that
\begin{multline*}
\left( \aver{2Q_i^t} \big| B_{Q_i^t}f(x) - B_{Q}f(x) \big|^{q_0} dx \right)^{1/q_0}
=
\left( \aver{2Q_i^t} \big| A_{Q_i^t}\,B_Qf(x) \big|^{q_0} dx \right)^{1/q_0}
\\
\le
\sum_{k\ge 2}
\alpha_k\left( \aver{2^kQ_i^t} |B_Qf(x)|^{p_0} dx \right)^{1/p_0}
=
\sum_{k\le k_i}\dots+\sum_{k>k_i}\dots
=I_1+I_2.
\end{multline*}
For $I_1$ we notice that $2^{k_i}Q_i^t\subset 2\,Q$ and then $B_Qf(x)=G(x)$ for $x\in 2^{k}Q_i^t$, $k\le k_i$ and therefore
\begin{align*}
I_1
\leq
\sum_{k\le k_i}
\alpha_k\left( \aver{2^kQ_i^t} | G(x)|^{p_0} dx \right)^{1/p_0}
\lesssim
t\,\sum_{k\ge 2}\alpha_k
\lesssim t.
\end{align*}
Here the second inequality follows as in \eqref{Whitney-prop} and we have use that $\{\alpha_k\}_{k\ge 2}\in\ell^1$.

For $I_2$ we use that $2^{k}\,Q_i^t\subset 2^{k}\,Q$, hence \eqref{hyp:BQ:2k-Q} yields
\begin{align}
I_2
&\le
\sum_{k> k_i}\alpha_k2^{k_i\,n/p_0}\left( \aver{2^{k} Q} \big| B_Qf(x) \big|^{p_0} dx \right)^{1/p_0}
\le
\sum_{k\ge 2}\alpha_k 2^{k\,n/p_0}a(2^{k} Q)
\le
\tilde{a}(Q). \label{eq:I2}
\end{align}

Gathering all the obtained estimates the proof is completed.

\qed

\subsection{Proof of Theorem \ref{theorem:Lp-alternative-bis}}
Proceeding as in the previous proof we set $\tilde{a}(Q)=C\,a(2\,Q)$ which is assumed to be finite and define $G$. Note that we have the analogue of \eqref{G-Lp0} as a consequence of \eqref{hyp:BQ:2k-Q-bis}:
$$
\left(\aver{2\,Q} G^{p_0}\,dx\right)^{1/p_0}
=
\left(\aver{2\,Q} |B_Q f(x)|^{p_0}\,dx\right)^{1/p_0}
\le
a(2Q)
\le \tilde{a}(Q)
.
$$
We obtain the very same good-$\lambda$ inequality as follows. In the case $t\gtrsim \tilde{a}(Q)$ we have to estimate $I$ and $II$ with the same definition. To estimate $I$ we proceed as before and $\Sigma_1$ is controlled in the same way. For $\Sigma_2$ we observe that \eqref{hyp:BQ:2k-Q-bis} also implies \eqref{estimate-comm-sigma2} and then the rest of argument is the same.

In this case the estimate for $II$ is much easier. We use \eqref{eqn:commu-reduction} (which follows from
\eqref{replace-comm}), \eqref{eqn:localization:AQ} and the fact that $2\,Q_i^t\subset 2\,Q$ to obtain that, for every $x\in 2\,Q_i^t$,
$$
B_{Q_i^t}f(x) - B_{Q}f(x)
=
-A_{Q_i^t}\,B_Qf(x)
=
-A_{Q_i^t}\big((B_Q f)\,\chi_{2\,Q_i^t}\big)(x)
=
-A_{Q_i^t}(G\,\chi_{2\,Q_i^t})(x)
$$
Then \eqref{AQ:on:p-q} and \eqref{Whitney-prop} yield
\begin{multline*}
\left( \aver{2Q_i^t} \big| B_{Q_i^t}f(x) - B_{Q}f(x) \big|^{q_0} dx \right)^{1/q_0}
\le
\alpha_2
\left( \aver{2\,Q_i^t} |G(x)|^{p_0} dx \right)^{1/p_0}
\le
10^{n/p_0}\,\alpha_2\,t.
\end{multline*}
Gathering the obtained estimates the proof is complete.
\qed

\subsection{Proof of Theorem \ref{theorem:Lp-w}}
The proof follows the scheme of the proof of Theorem \ref{theorem:Lp} and therefore we only point out the main changes.

Let us first observe that without loss of generality we can assume that $q_0<\infty$: indeed we can replace $q_0$ by $\tilde{q}_0$, chosen in such a way that $q<\tilde{q}_0<\infty$ and $w\in RH_{(\tilde{q}_0/q)'}$ since we have that $O(p_0,q_0)$ implies $O(p_0,\tilde{q}_0)$ and also that if $w\in RH_1=A_\infty$ then $w\in RH_{r}$ for some $r$ very close to $1$. Next we  recall the following well-known fact about reverse H\"older classes: $w\in RH_{(q_0/q)'}$ implies that there exists $1<r<q_0/q$ such that $w\in RH_{r'}$. We will use this  below.

We first note that under the present assumptions Lemma \ref{lemma:BQ2-hyp} holds. On the other hand, we need to adapt the proof of Lemma \ref{lemma:BQ2-reduction} to our current settings. Notice that \eqref{eq:AQBQ} holds with $q_0$  in place of $q$ on the left hand. This and the fact that $w\in RH_{(q_0/q)'}$ imply$$
\|A_Q B_Q f\|_{L^{q,\infty}(w),Q}
\le
\|A_Q B_Q f\|_{L^{q}(w),Q}
\le
C_w\,\|A_Q B_Q f\|_{L^{q_0},Q}
\lesssim
\sum_{k=1} ^\infty \tilde{\gamma}_k\,a(2^{k+1}\,Q).
$$
Then we obtain that \eqref{conclusion:theor-Lp-w} follows from
\begin{equation} \label{desired:BQ^2-w}
\|B_Q^2 f\|_{L^{q,\infty}(w), Q}
\lesssim
\sum_{k=1} ^\infty \tilde{\gamma}_k\,a(2^{k+1}\,Q),
\end{equation}
with $\tilde{\gamma}_k\gtrsim \alpha_{k+1}$ for $k\ge 1$.

As in the proof of Theorem \ref{theorem:Lp} we define $\tilde{a}$ so that the right hand side of \eqref{desired:BQ^2-w} is $\tilde{a}(2\,Q)$. We take the same function $G$ and observe that \eqref{G-Lp0} and \eqref{Omega-t} hold. We are going to show that for some large enough $s>1$ the following weighted good-$\lambda$ inequality holds:
given $0<\lambda<1$, for all $t>0$,
\begin{equation}\label{good-lambda-w}
w(\Omega_{s\,t}\cap Q) \leq c\,\left[\left(\frac{\lambda}{s}\right)^{p_0\,q/q_0}+s^{-q_0/r}\right]\, w(\Omega_t\cap Q) + c\,
\bigg(\frac{c_0\,\tilde{a}(2\,Q)}{\lambda\,t}\bigg)^{q}\,w(Q).
\end{equation}
Once this is obtained the desired estimate follows as before where in this case we have to pick $s$ large enough and $\lambda$ small enough such that
$$
c\, \left[s^{q-p_0\,q/q_0}\lambda^{p_0\,q/q_0}+s^{q-q_0/r}\right] \leq \frac{1}{2}.
$$
Note that this can be done since $1<r<q_0/q$ yields $q-q_0/r<0$.

Next we obtain \eqref{good-lambda-w}. The case $0<t\lesssim \tilde{a}(2\,Q)$ is trivial. Otherwise, for $t\gtrsim \tilde{a}(2\,Q)$ we proceed as before replacing the Lebesgue measure by $w$ and we conclude that
\begin{multline*}
w(\Omega_{s\,t}\cap Q) \leq
\sum_{i: Q_i^t\subset Q} w\big( \big\{x\in Q_i^t: M_{p_0} \big(|B_{Q_i^t}^2f|\, \chi_{2\,Q_i^t}\big)(x)> s\,t/4\big\} \big)
\\
 +
\sum_{i: Q_i^t\subset Q} w\big( \big\{x\in Q_i^t: M_{p_0} \big(|B_{Q_i^t}^2f-B_Q^2f|\, \chi_{2\,Q_i^t}\big)(x)> s\,t/4\big\} \big)
=I+II.
\end{multline*}

We first estimate $II$. We claim that Proposition \ref{prop:II} holds under the present assumptions. Indeed, a careful examination of the proof shows that the only estimate that needs to be checked is the last inequality in \eqref{use-D0} where we used that $\tilde{a}\in D_q$ to obtain that $\tilde{a}\in D_0$. Here, we have $\tilde{a}\in D_q(w)$ and this also implies $\tilde{a}\in D_0$ since $w\,dx$ is a doubling measure (see the comments after Definition \ref{def:Dr}).

Next  we use \eqref{eq:RH-sets} with the exponent $p=r'$ given above, the strong-type $(q_0,q_0)$ (with respect to the Lebesgue measure) of $M_{p_0}$, Proposition \ref{prop:II} and the fact that $\tilde{a}(2\,Q)\lesssim t$:
\begin{align*}
II
&=
\sum_{i: Q_i^t\subset Q} \frac{w\big( \big\{x\in Q_i^t: M_{p_0} \big(|B_{Q_i^t}^2f-B_Q^2f|\, \chi_{2\,Q_i^t}\big)(x)> s\,t/4\big\} \big)}{w(Q_i^t)}\,w(Q_i^t)
\\
&\lesssim
\sum_{i: Q_i^t\subset Q} \left(\frac{\big| \big\{x\in Q_i^t: M_{p_0} \big(|B_{Q_i^t}^2f-B_Q^2f|\, \chi_{2\,Q_i^t}\big)(x)> s\,t/4\big\} \big|}{|Q_i^t|}\right)^{1/r}\,w(Q_i^t)
\\
&\lesssim
\sum_{i: Q_i^t\subset Q} \left(\frac{1}{(s\,t)^{q_0}}\,\aver{2\,Q_i^t}  \big|B_{Q_i^t}^2f-B_Q^2f\big|^{q_0}\,dx\right)^{1/r}\,w(Q_i^t)
\\
&\lesssim
\sum_{i: Q_i^t\subset Q} \left(\frac{\big(\tilde{a}(2\,Q)+t\big)^{q_0}}{(s\,t)^{q_0}}\right)^{1/r}\,w(Q_i^t)
\\
&\lesssim
s^{-q_0/r}\,w(\Omega_t\cap Q).
\end{align*}
To complete the proof of \eqref{good-lambda-w} we just need to use the analogue of Proposition \ref{prop:I} given next. \qed

\begin{proposition} \label{prop:I:w}
The term $I$ can be estimated as follows:
$$
I \lesssim \left(\frac{\lambda}{s}\right)^{p_0\,q/q_0} w(\Omega_t \cap Q) + \left( \frac{\tilde{a}(2\,Q)}{\lambda t} \right)^{q} w(Q).
$$
\end{proposition}

\begin{proof}[Proof of Proposition \ref{prop:I:w}]
We follow the proof of Proposition  \ref{prop:I}. We take the same sets $\Gamma_1$, $\Gamma_2$ and define the corresponding sums $\Sigma_1$, $\Sigma_2$. For $\Sigma_1$ we use \eqref{eq:RH-sets} with  $p=(q_0/q)'$, that the maximal function is of weak-type $(1,1)$, the definition of the set $\Gamma_1$ and then the fact that the cubes $Q_i^t$ are pairwise disjoint and they are all contained in $\Omega_t \cap Q$:
\begin{align*}
\Sigma_1
&
=
\sum_{i: Q_i^t\in\Gamma_1} \frac{w\big( \big\{x\in Q_i^t: M_{p_0} \big(|B_{Q_i^t}^2f|\, \chi_{2\,Q_i^t}\big)(x)> s\,t/4\big\} \big)}{w(Q_i^t)}\,w(Q_i^t)
\\
&\lesssim
\sum_{i: \in\Gamma_1} \left(\frac{\big| \big\{x\in Q_i^t: M_{p_0} \big(|B_{Q_i^t}^2f|\, \chi_{2\,Q_i^t}\big)(x)> s\,t/4\big\} \big|}{|Q_i^t|}\right)^{q/q_0}\,w(Q_i^t)
\\
&
\lesssim
\sum_{i: \in\Gamma_1} \left(
\frac{1}{(s\,t)^{p_0}}\,
\aver{2\,Q_i^t}
|B_{Q_i^t}^2f|^{p_0}\, dx
\right)^{q/q_0}\,w(Q_i^t)
\\
&\lesssim
\left(\frac{\lambda}{s}\right)^{p_0\,q/p_0} \sum_{i\in \Gamma_1} w(Q_i^t)\\
&\lesssim
\left(\frac{\lambda}{s}\right)^{p_0\,q/p_0} w(\Omega_t \cap Q).
\end{align*}
This corresponds to the first part of the desired inequality.

For the indices $i\in \Gamma_2$ we use Lemma \ref{lemma:BQ2-hyp}, which as mentioned above holds in the present situation, and obtain \eqref{eq:atilde:Dq}.
Then, using that  $\tilde{a}\in D_{q}(w)$ and that $\{Q_i^t\}_i\subset Q\subset 2\,Q$ is a collection of pairwise disjoint cubes, we conclude that
$$
\Sigma_2
\leq \sum_{i\in \Gamma_2} w(Q_i^t)
\leq
\sum_{i\in \Gamma_2} \left(\frac{\tilde{a}(Q_i^t)}{\lambda t}\right)^{q} w(Q_i^t)
\lesssim
\left(\frac{\tilde{a}(2\,Q)}{\lambda t}\right)^{q} w(Q).
$$
\end{proof}

\subsection{Proof of Theorem \ref{theor:exponential}, (\ref{conclusion:theor-expo})}
We use the previous ideas and combine them with \cite{Jim}, \cite{MP}. The proof follows that of Theorem \ref{theorem:Lp-alternative} (working with $B_Q$ in place of $B_Q^2$) and thus we skip some details. We can assume that $\|a\|_{D_\infty}=1$. Indeed, if we set $\hat{a}(Q)=\sup_{P \subset Q} a(P)$ we trivially have $a(Q)\le \hat{a}(Q)\le \|a\|_{D_\infty}\,a(Q)$ and $\|\hat{a}\|_{D_\infty}=1$. Thus we can work with $\hat{a}$ and the resulting estimate will immediately imply that for $a$.

Set $\tilde{a}(Q)$ equal to the right hand side of \eqref{conclusion:theor-expo} and assume that it is finite. Notice that trivially $\tilde{a}\in D_\infty$ with $\|\tilde{a}\|_{D_\infty}=1$. Let
$G(x)=|B_Q f(x)|\,\chi_{2\,Q}(x)$ and $\Omega=\{x\in \re^n: M_{p_0} G(x)>\tilde{a}(Q)\}$. Then, by taking the implicit constant in \eqref{conclusion:theor-expo} large enough, we have
$$
|\Omega|
\lesssim \frac1{\tilde{a}(Q)^{p_0}}\int_{2\,Q} |B_Q f(x)|^{p_0}\,dx
\lesssim
\left(\frac{{a(2Q)}}{\tilde{a}(Q)}\right)^{p_0}\,|Q|
\le
e^{-1}\,|Q|.
$$
This gives in particular that $\Omega$ is a proper subset of $\re^n$ and then we can cover it as before by a family of Whitney cubes $\{Q_i\}_i$ associated to the dyadic grid induced by $Q$. Let us write
$$
\varphi(t)
=\sup_{R\in\mathcal{Q}} \frac{|E(R,t)|}{|R|},
\qquad\quad
E(R,t)=\{x\in R: |B_R f(x)|>t\,\tilde{a}(R)\},
$$
where it is understood that $E(R,t)=\emptyset$ if $\tilde{a}(R)=\infty$.
If $t>1$ then for a.e. $x\in E(Q,t)$ we have that $x\in \Omega$. Thus,
$$
|E(Q,t)|
=
|E(Q,t)\cap \Omega|
=
\sum_i|E(Q,t)\cap Q_i|.
$$
We can restrict  the previous sum to those $Q_i$'s with $Q_i\cap Q\neq\emptyset$ (since $E(Q,t)\subset Q$) in which case we have $Q_i\subset Q$ (here we use that $Q_i$ are cubes in the dyadic grid induced by $Q$ and also that $|Q_i|\le |\Omega|<|Q|$). We claim that
\begin{equation}\label{pointwise-BQ-BQi}
\|B_Q f-B_{Q_i} f\|_{L^\infty(Q_i)}\le C_0\,\tilde{a}(Q).
\end{equation}
Assuming this momentarily, for a.e.~$x\in E(Q,t)\cap Q_i$  we have by the previous estimate
$$
t\,\tilde{a}(Q)
<|B_Q f(x)|
\le
|B_Q f(x)-B_{Q_i} f(x)|+|B_{Q_i}f(x)|
\le
C_0\,\tilde{a}(Q)+|B_{Q_i}f(x)|.
$$
Besides, since $a\in D_\infty$ with $\|a\|_{D_\infty}=1$ we have $\tilde{a}(Q_i)\le \tilde{a}(Q)$. Using all these, we have, for every $t>C_0$,
\begin{align*}
|E(Q,t)|
&\le
\sum_{i: Q_i\subset Q}|\{x\in Q_i: |B_{Q_i}f(x)|>(t-C_0)\tilde{a}(Q)\}
\\
&\le
\sum_{i: Q_i\subset Q}|\{x\in Q_i: |B_{Q_i}f(x)|>(t-C_0)\tilde{a}(Q_i)\}
\\
&\le
\varphi(t-C_0)\,\sum_{i: Q_i\subset Q} |Q_i|
\\
&\le
\varphi(t-C_0)\,|\Omega|
\\
&\le
\varphi(t-C_0)\,e^{-1}\,|Q|.
\end{align*}
Thus we have shown that $|E(Q,t)|/|Q|\le \varphi(t-C_0)\,e^{-1}$ for every $t>C_0$ and for every cube $Q$ for which $\tilde{a}(Q)<\infty$. Notice that the same estimates holds trivially if $\tilde{a}(Q)=\infty$ since in such a case $|E(Q,t)|=0$. Then we can take the supremum over all cubes an conclude that $\varphi(t)\le \varphi(t-C_0)\,e^{-1}$ for every $t>C_0$. Iterating this estimate and using that $\varphi(t)\le 1$ for every $t\ge 0$ we obtain $\varphi(t)\le e^{1-t/C_0}$ for all $t\ge 0$. With this in hand we can obtain the desired estimate. If $\tilde{a}(Q)=\infty$ there is nothing to prove. Otherwise, we have $|E(Q,t)|/|Q|\le \varphi(t)\le e^{1-t/C_0}$ for all $t\ge 0$ and therefore by taking $A=C_0(e+1)>C_0$ we have
\begin{multline*}
\aver{Q} \left(\exp\left(\frac{|B_Q f(x)|}{A\,\tilde{a}(Q)}\right)-1\right)\,dx
=
\int_0^\infty e^t\,\frac{|E(Q,A\,t)|}{|Q|}\,dt
\\
\le
e\,\int_0^\infty e^{-t\,(\frac{A}{C_0}-1)}\,dt
=
e\,\int_0^\infty e^{-e\,t}\,dt
=1.
\end{multline*}
This gives as desired $\|B_Q f\|_{\exp L,Q}\le A\,\tilde{a}(Q)$.

To complete the proof we need to show our claim \eqref{pointwise-BQ-BQi}. This is an $L^\infty$ analogue of Proposition \ref{prop:II}. Following the ideas in the proof of Theorem \ref{theorem:Lp-alternative} and using the commutative condition we clearly have
$$
B_{Q_i^t}f(x) - B_{Q}f(x)
=
B_{Q_i}(\Id-B_Q)f(x) - (\Id-B_{Q_i})B_Qf(x)
=
A_QB_{Q_i}f(x)-A_{Q_i}B_Q f(x),
$$
and we estimate each term in turn. We take $k_i$ as in \eqref{lado-i} so that \eqref{Q_it-Q} holds. For the first term we use that $Q_i\subset Q$ and \eqref{AQ:off:p-q:j=1} with $q_0=\infty$:
\begin{align*}
\|A_QB_{Q_i}f\|_{L^\infty(Q_i)}
\le
\|A_QB_{Q_i}f\|_{L^\infty(2\,Q)}
\le
\sum_{k\ge 2}\alpha_k\,\left(\aver{2^{k} Q} |B_{Q_i}f|^{p_0}\,dx \right)^{1/p_0}.
\end{align*}
Notice that \eqref{lado-i} and \eqref{Q_it-Q} imply
$$
2^k\,Q\subset 2^{k+k_i+2}\,Q_i\subset 2^{k+3}\,Q,
\qquad
|2^k\,Q|\approx |2^{k+k_i+2}\,Q_i|\approx |2^{k+3}\,Q|.
$$
This, \eqref{hyp:BQ:2k-Q} and the fact that $a\in D_\infty$ yield
\begin{multline*}
\|A_QB_{Q_i}f\|_{L^\infty(Q_i)}
\lesssim
\sum_{k\ge 2}\alpha_k\,\left(\aver{2^{k+k_i+2}\,Q_i} |B_{Q_i}f|^{p_0}\,dx \right)^{1/p_0}
\\
\le
\sum_{k\ge 2}\alpha_k\,a(2^{k+k_i+2}\,Q_i)
\le
\sum_{k\ge 2}\alpha_k\,a(2^{k+3}\,Q)
\le
\tilde{a}(Q).
\end{multline*}

Next, the second term is treated essentially as in the proof of Theorem \ref{theorem:Lp-alternative}:
we use \eqref{AQ:off:p-q:j=1} with $q_0=\infty$
\begin{multline*}
\|A_{Q_i}\,B_Qf\|_{L^\infty(Q_i)}
\le
\|A_{Q_i}\,B_Qf\|_{L^\infty(2\,Q_i)}
\le
\sum_{k\ge 2}\alpha_k\,\left(\aver{2^{k}\,Q_i} |B_{Q}f|^{p_0}\,dx \right)^{1/p_0}
\\
=
\sum_{k\le k_i}\dots+\sum_{k>k_i}\dots
=I_1+I_2.
\end{multline*}
For $I_1$, using that $2^{k_i}Q_i^t\subset 2\,Q$ we obtain $B_Qf(x)=G(x)$ for $x\in 2^{k}Q_i^t$, $k\le k_i$ and therefore
\begin{align*}
I_1
\leq
\sum_{k\le k_i}
\alpha_k\left( \aver{2^kQ_i^t} | G(x)|^{p_0} dx \right)^{1/p_0}
\lesssim
\tilde{a}(Q)\,\sum_{k\ge 2}\alpha_k
\lesssim \tilde{a}(Q).
\end{align*}
The second inequality follows from \eqref{Whitney-prop} taking into account that here the level sets are done with $t=\tilde{a}(Q)$, and we have use that $\{\alpha_k\}_{k\ge 2}\in\ell^1$.

For $I_2$ we use that for every if $k>k_i$ we have $2^{k}\,Q_i\subset 2^{k-k_i+1}\,Q$ with $|2^{k}\,Q_i|\approx |2^{k-k_i+1}\,Q|$. Hence \eqref{hyp:BQ:2k-Q} and $a\in D_\infty$ yields
\begin{multline*}
I_2
\le
\sum_{k> k_i}\alpha_k\left( \aver{2^{k-k_i+1} Q} \big| B_Q f(x) \big|^{p_0} dx \right)^{1/p_0}
\\
\le
\sum_{k>k_i}\alpha_k a(2^{k-k_i+1} Q)
\lesssim
\sum_{k\ge 2}\alpha_k a(2^{k} Q)
\le
\tilde{a}(Q).
\end{multline*}

Gathering all the obtained estimates the proof is completed.

\qed

\subsection{Proof of Theorem \ref{theor:exponential}, (\ref{conclusion:theor-expo:w})}
The proof is essentially the same and we only point out the few changes. Since $w\in A_\infty$ we have that $w\in RH_r$ for some $r>1$ and therefore
\eqref{eq:RH-sets} gives $w(E)/w(Q)\le C_w(|E|/|Q|)^{1/r'}$ for every $E\subset Q$. Taking the constant in \eqref{conclusion:theor-expo:w} large enough we have as before
$
|\Omega|\le
e^{-r'(1+\log C_w)}\,|Q|.
$
Next we define a new function $\varphi(t)=\sup_R w(E(R,t))/w(R)$. Then proceeding as before and using \eqref{pointwise-BQ-BQi} we obtain, for every $t>C_0$,
\begin{multline*}
w(E(Q,t))
\le
\varphi(t-C_0)\,\sum_{i: Q_i\subset Q} w(Q_i)
=
\varphi(t-C_0)\,w( \Omega \cap Q)
\\
\le
\varphi(t-C_0)\,C_w\,w(Q)\,\left(\frac{|\Omega\cap Q|}{|Q|}\right)^{1/r'}\le
\varphi(t-C_0)\,e^{-1}\,w(Q).
\end{multline*}
From here the rest of the proof extends \textit{mutatis mutandis} with $dw$ replacing $dx$.

\subsection{Proofs related to the applications}

\subsubsection{Proof of Lemma \ref{lemma:D-q:EPI:w}}

The argument uses some ideas from \cite{JM} and runs parallel to the proof of \cite[Proposition 4.1]{BJM}, therefore we only give the main details.
It is well-known (see for instance \cite{GR} or \cite{Grafakos}) that given $w\in A_1$ there exists $0<\theta\le 1$ such that for every cube $Q$ and for every measurable subset $S\subset Q$ we have
\begin{equation}\label{w-A1-Ainfty}
\frac{|S|}{|Q|}\lesssim \frac{w(S)}{w(Q)}\lesssim \left(\frac{|S|}{|Q|}\right)^{\theta}.
\end{equation}
The first inequality follows at one from $w\in A_1$ and the second one uses that $w\in RH_{(1/\theta)'}$ for some $0<\theta\le 1$ ---notice that we allow $\theta=1$ to cover the unweighted case, i.e., $w\equiv 1$.
We write
$$
a(Q)=\sum_{k=0}^\infty \gamma_k\,a_0(2^k\,Q),\qquad\quad a_0(Q)=\ell(Q)\,\left(\aver{Q} h^s\, dw\right)^{1/s}.
$$
We may assume that $s\le q<s^*$ since the $D_q$ conditions are decreasing. Fix a cube $Q$ and a family $\{Q_i\}_i\subset Q$ of pairwise
disjoint cubes.  Then, Minkowski's inequality and $q\geq s$ yield
$$
\Big( \sum_{i} a(Q_i)^{q}\, w(Q_i) \Big)^{\frac1q}
\le
\sum_{k= 0}^\infty \gamma_k\,\Big(\sum_i \frac{\ell(2^k\,Q_i)^s \, w(Q_i)^{\frac{s}{q}}}{w(2^k\,Q_i)}\,\int_{2^k\,Q_i} h^s\,dw
\Big)^{\frac1s}
:=
\sum_{k= 0}^\infty \gamma_k \,I_k^{\frac1s}.
$$
If $k=0$, we use  $s\leq q<s^*$ and \eqref{w-A1-Ainfty} (indeed the left hand side inequality) to obtain for every $i$
$$
\left(\frac{\ell(Q_i)}{\ell(Q)}\right)^s \left(\frac{w(Q)}{w(Q_i)}\right)^{1-\frac{s}{q}}
\lesssim
\left(\frac{\ell(Q_i)}{\ell(Q)}\right)^s \left(\frac{|Q|}{|Q_i|}\right)^{1-\frac{s}{q}}
=
\left(\frac{|Q_i|}{|Q|}\right)^{s\,(\frac1q-\frac{1}{s^*})}
\le 1
$$
Hence, since the cubes $Q_i\subset Q$ are pairwise disjoint, we get
$$ I_0 \lesssim
\frac{\ell(Q)^s \, w(Q)^{\frac{s}{q}}}{w(Q)}\,
\sum_i \int_{Q_i} h^s\,dw
\le
\frac{\ell(Q)^s \, w(Q)^{\frac{s}{q}}}{w(Q)}\,
\int_{Q} h^s\,dw
=
a_0(Q)^s\,w(Q)^{\frac{s}{q}}
.
$$

For $k\ge 1$ we arrange the cubes according to their sidelength and use an estimate of the overlap: given $l\ge 1$, write $E_l=\{Q_i:2^{-l}\, \ell(Q)<\ell(Q_i)\le  2^{-l+1}\, \ell(Q)\}$. As it is obtained in \cite[Lemma 4.3]{JM} we have that $2^k\,Q_i\subset 2^{\max\{k-l+1,0\}+1}Q$. Let us observe that $\# E_l\lesssim 2^{l\,n}$ since all the cubes in $E_l$ have comparable size, are disjoint and contained in $Q$. On the other hand it is shown in \cite[Lemma 4.3]{JM} that the overlap is at most $C\,2^{k\,n}$. Therefore we have $\sum_{Q_i\in E_l} \chi_{2^k\,Q_i}\lesssim 2^{n\min\{l,k\}}\,\chi_{2^{\max\{k-l+1,0\}+1}Q}$. Using this, $s<n$ and $s\leq q<s^*$:
$$
I_k
=
\sum_{l=1}^\infty \sum_{Q_i\in E_l}\frac{\ell(2^k\,Q_i)^s \, w(Q_i)^{\frac{s}{q}}}{w(2^k\,Q_i)}\,\int_{2^k\,Q_i} h^s\,dw
=
\sum_{l=1}^{k+1}\cdots+\sum_{l=k+2}^\infty\cdots
=\Sigma_1+\Sigma_2.
$$
For $\Sigma_1$, using \eqref{w-A1-Ainfty} and the previous observations we have
\begin{align*}
\Sigma_1
&=
\sum_{l=1}^{k+1}\frac{\ell(2^{k-l+2}\,Q)^s\,w(Q)^{\frac{s}{q}}}{w(2^{k-l+2}\,Q)}
\sum_{Q_i\in E_l}
\bigg(\frac{\ell(2^{k}\,Q_i)}{\ell(2^{k-l+2}\,Q)}\bigg)^s\,
\bigg(\frac{w(Q_i)}{w(Q)}\bigg)^{\frac{s}{q}}\,
\frac{w(2^{k-l+2}\,Q)}{w(2^k\,Q_i)}
\\
&\hskip8cm\times
\int_{2^k\,Q_i} h^s\,dw
\\
&
\lesssim
w(Q)^{\frac{s}{q}}\sum_{l=1}^{k+1}
\frac{\ell(2^{k-l+2}\,Q)^s}{w(2^{k-l+2}\,Q)}\,2^{-l\,\theta\,n\,\frac{s}{q}}\,\sum_{Q_i\in E_l} \int_{2^k\,Q_i} h^s\,dw
\\
&\lesssim
w(Q)^{\frac{s}{q}}\sum_{l=1}^{k+1}
\frac{\ell(2^{k-l+2}\,Q)^s}{w(2^{k-l+2}\,Q)}\,2^{-l\,\theta\,n\,\frac{s}{q}}\,2^{n\,l}\int_{2^{k-l+2}\,Q} h^s\,dw
\\
&
=
w(Q)^{\frac{s}{q}}\sum_{l=1}^{k+1}2^{l\,n\,(1-\theta\,\frac{s}{q})}\,a_0(2^{k-l+2}\,Q)^s
\\
&
=
w(Q)^{\frac{s}{q}}\,2^{k\,n\,(1-\theta\,\frac{s}{q})} \sum_{l=1}^{k+1}2^{-l\,n\,(1-\theta\,\frac{s}{q})}\,a_0(2^{l}\,Q)^s.
\end{align*}
Analogously, for $\Sigma_2$,  using the fact that $s\le q<s^*$ we obtain
\begin{align*}
\Sigma_2
&=
\frac{\ell(2\,Q)^s\,w(Q)^{\frac{s}{q}}}{w(2\,Q)}
\,
\sum_{l=k+2}^\infty\sum_{Q_i\in E_l}
\bigg(\frac{\ell(2^{k}\,Q_i)}{\ell (2\,Q)}\bigg)^s\,
\bigg(
\frac{w(Q_i)}{w(2^k\,Q_i)}\,\frac{w(2\,Q)}{w(Q)}
\bigg)^{\frac{s}{q}}
\\
&\hskip8cm\times\bigg(\frac{w(2\,Q)}{w(2^k\,Q_i)}\bigg)^{1-\frac{s}{q}}\,\int_{2^k\,Q_i} h^s\,dw
\\
&\lesssim
\frac{\ell(2\,Q)^s\,w(Q)^{\frac{s}{q}}}{w(2\,Q)}
\,2^{k\,s\,(1+n\,\frac{1-\theta}{q})}\,2^{-k\,n}
\sum_{l=k+2}^\infty 2^{-l\,(s+n\,\frac{s}{q}-n)} \sum_{Q_i\in E_l}\int_{2^k\,Q_i} h^s\,dw
\\
&
\lesssim
\frac{\ell(2\,Q)^s\,w(Q)^{\frac{s}{q}}}{w(2\,Q)}\,\left(\int_{2\,Q} h^s\,dw\right)\,
2^{k\,s\,(1+n\,\frac{1-\theta}{q})}
\sum_{l=k+2}^\infty 2^{-l\,(s+n\,\frac{s}{q}-n)}
\\
&\lesssim
w(Q)^{\frac{s}{q}}\,2^{k\,n\,(1-\theta\,\frac{s}{q})}\,a_0(2\,Q)^s.
\end{align*}
Let us note that the last quantity corresponds to the term $l=1$ in the estimate for $\Sigma_1$. Gathering the obtained inequalities we conclude that
\begin{multline*}
\Big( \sum_{i} a(Q_i)^{q}\, w(Q_i) \Big)^{\frac1q}
\\
\lesssim
\gamma_0\,a_0(Q)\,w(Q)^{\frac1q}+w(Q)^{\frac1q}\sum_{k=1}^\infty \gamma_k\, 2^{k\,n\,(\frac1s-\frac{\theta}q)}\,\sum_{l=1}^{k+1} 2^{-l\,n\,(\frac1s-\frac{\theta}{q})} a_0(2^l\,Q)
\\
\le
\sum_{k=0}^\infty\bar{\gamma}_k\,a_0(2^k\,Q)\,w(Q)^{\frac1q}
=
\bar{a}(Q)\,w(Q)^{\frac1q},
\end{multline*}
where $\bar{\gamma}_0=C\,\gamma_0$ and $\bar{\gamma}_k=2^{-k\,n\,(\frac1s-\frac{\theta}q)}\,\sum_{l=k-1}^\infty \gamma_l\,2^{l\,n\,(\frac1s-\frac{\theta}{q})}$, $k\ge 1$ ---notice that here we implicitly use that $\gamma_k$ is a fast decay sequence, since otherwise the coefficient $\bar{\gamma}_k$ would be infinity.
This completes the proof since we have shown that $(a,\bar{a})\in D_q(w)$. \qed

\begin{remark}
We would like to call the reader's attention to the fact that, in the previous argument, it was crucial that
$s\leq q<s^*$, since otherwise the geometric sum for the terms $l\ge k+2$ diverges.
\end{remark}

\subsubsection{Proof of Proposition \ref{prop:off-semi}}
We fix $N\ge 1$. Abusing the notation we write $A_Q$ and $B_Q$ in place of $A_{Q,N}$ and $B_{Q,N}$.
Property $(a)$ follows at once after expanding $B_Q$ and using the semigroup property $e^{-t\,L}e^{-s L}=e^{-(t+s)\,L}$. Regarding $(b)$ we invoke the theory of off-diagonal estimates developed in \cite{AM2}. Our assumption \eqref{ass:offdiag}  implies that $\{e^{-t\,L}\}_{t>0}$ satisfies $L^{p_0}(\re^n)-L^{q_0}(\re^n)$ full off-diagonal estimates which are equivalent to the $L^{p_0}(\re^n)-L^{q_0}(\re^n)$ off-diagonal estimates on balls \cite[Section 3.1]{AM2}, and these imply uniform boundedness of $e^{-t\,L}$ on $L^{p_0}(\re^n)$, \cite[Theorem 2.3]{AM2}. This yields $(b)$ after expanding $B_Q$. To see $(c)$ we expand $A_Q$ and observe that it suffices to prove that for every cube, $e^{-l\,\ell(Q)^2L}$ verifies \eqref{AQ:on:p-q} and \eqref{AQ:off:p-q} for every $1\le l\le N$. Fixed such $l$ we use \eqref{ass:offdiag} to obtain \eqref{AQ:on:p-q}:
\begin{align*}
\left(\aver{2Q} \left|e^{-l\,\ell(Q)^2 L} (f\,\chi_{4\,Q})\right|^{q_0}\,dx\right)^{1/q_0}
& \le
C\,|2Q|^{-1/q_0}\,(l\,\ell(Q)^2)^{-\frac{n}{2}\left(\frac{1}{p_0}-\frac{1}{q_0}\right)}\,\left(\int_{4Q} |f|^{p_0}\,dx\right)^{1/p_0}
\\
& \lesssim C\,\left(\aver{4Q} |f|^{p_0}\,dx \right)^{1/p_0}.
\end{align*}
On the other hand, using again \eqref{ass:offdiag} we show \eqref{AQ:off:p-q}:
\begin{align*}
&\left(\aver{2^{j}Q} \left|e^{-l\,\ell(Q)^2 L} (f\,\chi_{\re^n\setminus 2^{j+1}\,Q})\right|^{q_0}\,dx\right)^{1/q_0}
\\
&\qquad\le
\sum_{k\ge 2} \left(\aver{2^{j}Q} \left|e^{-l\,\ell(Q)^2 L} (f\,\chi_{2^{k+j}Q\setminus 2^{k+j-1}Q})\right|^{q_0}\,dx\right)^{1/q_0}
\\
&\qquad\lesssim
\sum_{k\ge 2} |2^jQ|^{-1/q_0} (l\,\ell(Q)^2 )^{-\frac{n}{2}\left(\frac{1}{p_0}-\frac{1}{q_0}\right)}\,e^{-c\,\frac{4^{j+k}\,\ell(Q)^2}{l\,\ell(Q)^2}}\, \left(\int_{2^{k+j}Q} |f|^{p_0} \right)^{1/p_0}
\\
&\qquad\lesssim \sum_{k\geq 2} e^{-c\,4^{k+j}} \left(\aver{2^{k+j} Q} |f|^{p_0}\,dx \right)^{1/p_0}.
\end{align*}
As noted above the constant $c$ is irrelevant, provided it remains positive, and therefore $c$ may change from line to line.

The proof of \eqref{BR-AQ:off:p-q} is more delicate and we need to exploit the fact that we have off-diagonal decay for $(t\,L)^k\,e^{-t\,L}$ with $0\le k\le N$ and that $N\ge n/(2\,q_0)$.
As before we expand $A_Q$ and so it suffices to prove \eqref{BR-AQ:off:p-q} for the operator
$$
S_{Q,R}:= \left(\Id-e^{-\ell(R)^2\, L}\right)^N e^{-l\,\ell(Q)^2\,L},
$$
with $1\le l\le N$ fixed and where $R\subset Q$. Let us now point out that we can factor out some power of $L$ in the previous operator. Indeed, the strong continuity (in $L^{p_0}$) of the semigroup at $t=0$ (more precisely $e^{-tL} \to {\mathcal I}$ for $t\to 0$) with Remark \ref{rem:continuity} give us that, for every $s>0$,
\begin{equation}
sL\left(\aver{[0,s]} e^{-\lambda\, L} d\lambda \right)
=
\int_0^s L e^{-\lambda\, L} d \lambda = \Id-e^{-s\, L},
\label{eq:Us}
\end{equation}
where this equality holds in the sense of $L^p(\re^n)$-bounded operators, for every finite $p\in[p_0,q_0]$, $p<\infty$.
Hence for $s>0$, we define
$$
U_s:= \left(\aver{[0,s]} e^{-\lambda\, L} d\lambda\right)^N
$$
and $S_{Q,R}$ can be written as follows:
\begin{align*}
S_{Q,R}
&=
\left(\frac{\ell(R)}{\sqrt{l}\ell(Q)}\right)^{2N}  (l\,\ell(Q)^2\,L)^N\,e^{-l\,\ell(Q)^2\,L}\,U_{\ell(R)^2}.
 \end{align*}
Let us observe that writing the operator in this way we have obtained an extra factor $(\ell(R)/\ell(Q))^{2\,N}$, which is small since $R\subset Q$.

We have the following off-diagonal estimates for $U_s$, the proof is given below.

\begin{lemma} \label{lemma:Us}
The family of operator $\{U_s\}_{s>0}$ satisfies $L^{p_0}(\re^n)-L^{p_0}(\re^n)$ off-diagonal estimates: for all closed sets $E,F$ and functions $f$ supported in $F$
\begin{equation} \label{off-Us}
\left(\int_E \left|U_s f\right|^{p_0} dx\right)^{1/p_0} \lesssim e^{-c\frac{d(E,F)^2}{s}} \left(\int_{F} |f|^{p_0} dx\right)^{1/p_0}.
\end{equation}
\end{lemma}

Since $R\subset Q$ we can find a unique $j\ge 0$ such that $2^j\ell(R)\le \ell(Q)<2^{j+1}\,\ell(R)$. Then $2^jR\subset 2Q$ and $|2^jR|\approx |Q|$. Using the off-diagonal decay of $(l\,\ell(Q)^2\,L)^N\,e^{-l\,\ell(Q)^2\,L}$ we get
\begin{align*}
&\left(\aver{2R} |(l\,\ell(Q)^2\,L)^N\,e^{-l\,\ell(Q)^2\,L} h|^{q_0} dx\right)^{1/q_0}
\\
&\quad\le
\left(\aver{2R} |(l\,\ell(Q)^2\,L)^N\,e^{-l\,\ell(Q)^2\,L} (h\,\chi_{2^{j+2} R})|^{q_0} dx\right)^{1/q_0}
\\
&\qquad\quad+
\sum_{k=3}^\infty \left(\aver{2R} |(l\,\ell(Q)^2\,L)^N\,e^{-l\,\ell(Q)^2\,L} (h\,\chi_{2^{k+j}R\setminus 2^{k+j-1} R})|^{q_0} dx\right)^{1/q_0}
\\
&\quad\lesssim
|2R|^{-1/q_0}\,|2^{j+2} R|^{1/p_0}\,(l\,\ell(Q)^2)^{-\frac{n}{2}\left(\frac{1}{p_0}-\frac{1}{q_0}\right)}\,\left(\aver{2^{j+2}R} |h|^{p_0} dx\right)^{1/p_0}
\\
&
\qquad\quad+
\sum_{k=3}^\infty |2R|^{-1/q_0}\,|2^{k+j} R|^{1/p_0}\,(l\,\ell(Q)^2)^{-\frac{n}{2}\left(\frac{1}{p_0}-\frac{1}{q_0}\right)}\,e^{-c\frac{4^{k+j}\,\ell(R)^2}{l\,\ell(Q)^2}}\left(\aver{2^{k+j}R} |h|^{p_0} dx\right)^{1/p_0}
\\
&\quad\lesssim
\left(\frac{\ell(Q)}{\ell(R)}\right)^{\frac{n}{q_0}}\,
\sum_{k=3}^\infty e^{-c\,4^k}\,\left(\aver{2^{k}Q} |h|^{p_0} dx\right)^{1/p_0}.
\end{align*}
We apply this estimate with $h=U_{\ell(R)^2} f$ and \eqref{off-Us} and use that $N\ge n/(2\,q_0)$ to obtain
\begin{align*}
&\left(\aver{2R} \left|S_{Q,R} f\right|^{q_0} \right)^{1/q_0}
\lesssim
\left(\frac{\ell(R)}{\ell(Q)}\right)^{2N}  \left(\frac{\ell(Q)}{\ell(R)}\right)^{\frac{n}{q_0}}\,
\sum_{k=3}^\infty e^{-c\,4^k}\,\left(\aver{2^{k}Q} |U_{\ell(R)^2} f|^{p_0} dx\right)^{1/p_0}
\\
&\quad\le
\sum_{k=3}^\infty e^{-c\,4^k}\,\left(\aver{2^{k}Q} |U_{\ell(R)^2} (f\,\chi_{2^{k+1}Q})|^{p_0} dx\right)^{1/p_0}
\\
&\qquad\qquad+
\sum_{k=3}^\infty e^{-c\,4^k}\,\sum_{l=2}^\infty\left(\aver{2^{k}Q} |U_{\ell(R)^2} (f\,\chi_{2^{k+l}Q\setminus 2^{k+l-1}Q})|^{p_0} dx\right)^{1/p_0}
\\
&\quad\lesssim
\sum_{k=3}^\infty e^{-c\,4^k}\,\left(\aver{2^{k+1}Q} |f|^{p_0} dx\right)^{1/p_0}
+
\sum_{k=3}^\infty e^{-c\,4^k}\,\sum_{l=2}^\infty e^{-c\,\frac{4^{k+l}\,\ell(Q)^2}{\ell(R)^2}}\,
\left(\aver{2^{k+l}Q} |f|^{p_0} dx\right)^{1/p_0}
\\
&
\quad\lesssim
\sum_{k=4}^\infty e^{-c\,4^k}\,\left(\aver{2^{k+1}Q} |f|^{p_0} dx\right)^{1/p_0}.
\end{align*}
This completes the proof of \eqref{BR-AQ:off:p-q} for $S_{Q,R}$ with $R\subset Q$ and so for $B_RA_Q$.
\qed

\begin{proof}[Proof of Lemma \ref{lemma:Us}]
Off-diagonal estimates are stable under composition (see for instance \cite[Lemma 2.3]{Hof-Mar} or \cite[Sections 2.4, 3.1]{AM2}). Thus it suffices to obtain the desired estimate in the case $N=1$.

As observed before \eqref{ass:offdiag}  implies that $\{e^{-t\,L}\}_{t>0}$ satisfies $L^{p_0}(\re^n)-L^{q_0}(\re^n)$ off-diagonal estimates on balls \cite[Section 3.1]{AM2}, thus we have  $L^{p_0}(\re^n)-L^{p_0}(\re^n)$ off-diagonal estimates on balls (see \cite[Sections 2.1]{AM2}). These in turn are equivalent to the $L^{p_0}(\re^n)-L^{p_0}(\re^n)$ (full) off-diagonal estimates for $\{e^{-t\,L}\}_{t>0}$ by \cite[Section 3.1]{AM2}. This and Minkowski's inequality allow us to obtain that for all closed sets $E,F$ and functions $f$ supported in $F$
\begin{multline*}
\left(\int_E \left| \aver{[0,s]} e^{-\lambda\, L} f d\lambda  \right|^{p_0} dx\right)^{1/p_0}
\le
\aver{[0,s]} \left(\int_E |e^{-\lambda\, L} f|^{p_0} dx\right)^{1/p_0} d\lambda
\\
\lesssim
\aver{[0,s]} e^{-c\,\frac{d(E,F)^2}{\lambda}}\left(\int_F |f|^{p_0} dx\right)^{1/p_0} d\lambda
\le
e^{-c\,\frac{d(E,F)^2}{s}}\left(\int_F |f|^{p_0} dx\right)^{1/p_0}.
\end{multline*}
\end{proof}

\begin{remark} Let us point out that under some invertibility properties on the generator $L$, the computation done in \eqref{eq:Us} implies that
$$  \aver{[0,s]} e^{-\lambda\, L} d\lambda= (s\,L)^{-1}\big(\Id-e^{-s\,L}\big),
$$
which gives $U_s=(s\,L)^{-N}\big(\Id-e^{-s\,L}\big)^N$. For example, this is the case in the applications discussed in Subsection \ref{subsec:appli}, where $L$ is an elliptic second-order divergence form operator.
\end{remark}

\subsubsection{Proof of Lemma \ref{lemma:Poincare-Hyp-k}} The argument is a combination of ideas from \cite{JM} (see also the proof of Lemma \ref{lemma:D-q:EPI} above). Let us write
$$
a(Q)=\sum_{j=0}^\infty \eta_j\,a_0(2^j\,Q),\qquad\quad a_0(Q)=\ell(Q)\,\left(\aver{Q} h^s\, dx\right)^{1/s}.
$$
By subdividing the cube $2^kQ$ we have a family $\{Q_i\}_{i=1}^{2^{k\,n}}$ of disjoint cubes such that $\ell(Q_i)=\ell(Q)$ and $\cup_i Q_i=2^k Q$.
We claim that
\begin{equation}\label{eqn:Ds-equal-size}
\Big(\sum_i a(Q_i)^s\,|Q_i|\Big)^{\frac1s}
\lesssim
a(2^k\,Q)\,|2^k\,Q|^{\frac1s}.
\end{equation}
Notice that if we used Lemma \ref{lemma:D-q:EPI} above we would get the bigger functional $\bar{a}$ on the right hand side. Here we obtain a better estimate since all the cubes in our family $\{Q_i\}_i$ have the same sidelength.

Using \eqref{eqn:Ds-equal-size}, \eqref{eqn:hypo-EPI} and H\"older's inequality we have
\begin{multline*}
\Big(\int_{2^k\,Q} \big|\big(\Id-e^{-\ell(Q)^2 L }\big)^{N}f(x)\big|^p\,dx\Big)^{\frac1p}
=
\Big(\sum_i \int_{Q_i} \big|\big(\Id-e^{-\ell(Q_i)^2 L }\big)^{N}f(x)\big|^p\,dx\Big)^{\frac1p}
\\
\le
\Big(\sum_i a(Q_i)^p\,|Q_i|\Big)^{\frac1p}
\le
\Big(\sum_i a(Q_i)^s\,|Q_i|\Big)^{\frac1s}\,\Big(\sum_i |Q_i|\Big)^{\frac1{p\,(s/p)'}}
\lesssim
a(2^k\,Q)\,|2^k\,Q|^{\frac1p},
\end{multline*}
and this is the desired inequality.

We show our claim \eqref{eqn:Ds-equal-size}. By Minkowski's inequality we have
\begin{multline*}
\Big(\sum_i a(Q_i)^s\,|Q_i|\Big)^{\frac1s}
=
\Big(\sum_i \Big(\sum_{j=0}^\infty \eta_j\, a_0(2^jQ_i)\Big)^s\,|Q_i|\Big)^{\frac1s}
\\
\le
\sum_{j=0}^\infty \eta_j\, \Big(\sum_i a_0(2^jQ_i)^s\,|Q_i|\Big)^{\frac1s}
=
\sum_{j=0}^\infty \eta_j\, I_j.
\end{multline*}
For a fixed $j$ we have that $2^j\,Q_i\subset 2^{\max\{j,k\}+1}Q$. Let us observe that $\sum_{i} \chi_{2^j\,Q_i}\lesssim 2^{n\min\{k,j\}}\,\chi_{2^{\max\{j,k\}+1}Q}$ since we have $2^{k\,n}$ cubes and it is shown in \cite[Lemma 4.3]{JM} that there are at most $2^{n\,(j+4)}$ cubes $Q_i$ such that $2^j\,Q_i$ meets some fixed $2^j\,Q_{i_0}$. Hence,
\begin{multline*}
I_j
=
2^{j(1-\frac{n}s)}\,\ell(Q)\,\Big(\sum_i \int_{2^j\,Q_i}h^s\,dx\Big)^{\frac1s}
\lesssim
2^{j(1-\frac{n}{s})} 2^{\frac{n}{s}\min\{k,j\}}\,\ell(Q)\,\Big(\int_{2^{\max\{j,k\}+1}Q }h^s\,dx\Big)^{\frac1s}
\\
=
C\,
2^{-\max\{k-j,0\}}\,a_0(2^{\max\{j,k\}+1}Q)\,|2^k\,Q|^{\frac1s}.
\end{multline*}

Thus,
\begin{multline*}
\Big(\sum_i a(Q_i)^s\,|Q_i|\Big)^{\frac1s}
\lesssim
|2^k\,Q|^{\frac1s}\sum_{j=0}^\infty \eta_j\,2^{-\max\{k-j,0\}}a_0(2^{\max\{j,k\}+1}Q)
\\
=
|2^k\,Q|^{\frac1s}\Big(a_0(2^{k+1}\,Q)\,\sum_{j=0}^k \eta_j\,2^{j-k}
+
\sum_{j=k+1}^\infty \eta_j\,a_0(2^{j+1}Q)\Big)
\\
\lesssim
|2^k\,Q|^{\frac1s}\Big(a_0(2^{k+1}\,Q)+\sum_{j=k+1}^\infty\eta_{j-k+1}\,a_0(2^{j+1}\,Q)\Big)
\lesssim
|2^k\,Q|^{\frac1s}\,a(2^{k}Q)
\end{multline*}
where we have used that the sequence $\eta_k$ is quasi-decreasing and that $\eta_1>0$.\qed

\end{document}